\documentclass[12pt,reqno]{amsart}

\newcommand\version{July 8, 2026}

\usepackage{amsmath,amsfonts,amsthm,amssymb,amsxtra}
\usepackage{bbm} 
\usepackage{hyperref} 
\usepackage{color}

\newtheorem{theorem}{Theorem}[section]
\newtheorem{proposition}[theorem]{Proposition}
\newtheorem{lemma}[theorem]{Lemma}

\theoremstyle{definition}

\newtheorem{assumption}[theorem]{Assumption}

\theoremstyle{remark}

\newtheorem{remark}[theorem]{Remark}

\newcommand{\1}{\mathbbm{1}}

\renewcommand{\epsilon}{\varepsilon}

\newcommand{\N}{\mathbb{N}}

\renewcommand{\phi}{\varphi}
\newcommand{\R}{\mathbb{R}}

\newcommand{\Sph}{\mathbb{S}}

\DeclareMathOperator{\dist}{dist}

\DeclareMathOperator{\supp}{supp}
\DeclareMathOperator{\sgn}{sgn}
\DeclareMathOperator{\Tr}{Tr}

\begin{document}

\title[Pleijel's theorem for degenerate elliptic operators]{Pleijel's theorem for a class of \\ degenerate elliptic operators}

\author{Rupert L. Frank}
\address[Rupert L. Frank]{Mathe\-matisches Institut, Ludwig-Maximilians Universit\"at M\"unchen, The\-resienstr.~39, 80333 M\"unchen, Germany, and Munich Center for Quantum Science and Technology, Schel\-ling\-str.~4, 80799 M\"unchen, Germany, and Mathematics 253-37, Caltech, Pasa\-de\-na, CA 91125, USA}
\email{r.frank@lmu.de}

\author{Bernard Helffer}
\address[Bernard Helffer]{Laboratoire de Math\'ematiques Jean Leray, CNRS, Nantes Universit\'e, F44000. Nantes, France}
\email{Bernard.Helffer@univ-nantes.fr}

\thanks{
	\textit{Date:} \version \\
	\copyright\, 2026 by the authors. This paper may be reproduced, in its entirety, for non-commercial purposes.\\
	Partial support through US National Science Foundation grant DMS-1954995 (R.L.F.), as well as through the German Research Foundation through EXC-2111-390814868 and TRR 352-Project-ID 470903074 (R.L.F.) is acknowledged.}

\begin{abstract}
	We prove an asymptotic upper bound on the number of nodal domains of eigenfunctions of a class of degenerate elliptic operators. Our proof yields the same constant as in Pleijel's bound for the Dirichlet Laplacian. The operators considered include the Baouendi--Grushin operator and operators with ellipticity degenerating on the boundary.
\end{abstract}

\maketitle

\section{Introduction}

Two fundamental results in spectral geometry are the theorems of Courant and of Pleijel concerning the number of nodal domains of eigenfunctions. While here we are interested in the extension of these theorems to degenerate elliptic operators, it is worthwhile to recall them in the simplest possible setting of the Dirichlet Laplacian $-\Delta_\Omega^{\rm D}$ on an open set $\Omega\subset\R^d$. It is well-known that eigenfunctions $u$ are continuous in $\Omega$ and therefore the nodal domains can be defined as the connected components of the open set $\{ x\in\Omega:\ u(x) \neq 0\}$. Assuming the spectrum of $-\Delta_\Omega^{\rm D}$ to be discrete, we label the eigenvalues in nondecreasing order, counting multiplicities, as $\lambda_1\leq\lambda_2 \leq\ldots$. Moreover, for each $k\geq 1$ we let
$$
\nu_k := \sup_{u \ \text{eigenfunction for}\ \lambda_k} \#\{ \text{nodal domains of}\ u \} \,.
$$
Then Courant's theorem \cite{C} states that, if $\Omega$ is connected, then
\begin{equation}
	\label{eq:courant}
	\frac{\nu_k}{k} \leq 1\,.
\end{equation}
Moreover, explicit computations show that one has equality in dimension $d=1$. However, in dimension $d>1$, Courant's bound can be asymptotically improved. Indeed, for any open set $\Omega\subset\R^d$ of finite measure one has
\begin{equation}
	\label{eq:pleijel}
	\limsup_{k\to\infty} \frac{\nu_k}{k} \leq \gamma(\R^d) \,.
\end{equation}
For $d=2$ and sufficiently regular $\Omega$, this is due to Pleijel \cite{Pl}. His proof generalizes to the case of open sets of finite measure.

The constant in \eqref{eq:pleijel} is
\begin{equation}\label{eq:pl}
	\gamma(\R^d) := \left( C^{\rm FK}(\R^d) \right)^{-d/2} \left( \mathcal W(\R^d) \right)^{-1},
\end{equation}
where $C^{\rm FK}(\R^d)$ is the (Euclidean) Faber--Krahn constant and $\mathcal W(\R^d)$ is the (Euclidean) Weyl constant. More precisely, $C^{\rm FK}(\R^d)$ is the largest constant such that for any open set $\omega\subset\R^d$ of finite measure and any $u\in H^1_0(\omega)$,
\begin{equation}
	\label{eq:fkeucl}
	\int_\omega |\nabla u|^2\,dx \geq C^{\rm FK}(\R^d) \, |\omega|^{-\frac2d} \int_\omega |u|^2\,dx \,,
\end{equation}
and the constant $\mathcal W(\R^d)$ is such that for any open set $\Omega\subset\R^d$ of finite measure the eigenvalue counting function $N(\lambda)$ of the Dirichlet Laplacian in $\Omega$ satisfies
\begin{equation}
	\label{eq:weyleucl}
	\lim_{\lambda\to\infty} \lambda^{-\frac d2} N(\lambda) = \mathcal W(\R^d) \,|\Omega| \,.
\end{equation}
Explicit expressions for $C^{\rm FK}(\R^d)$ and $\mathcal W(\R^d)$ in terms of Gamma functions and zeros of special functions are available. It is known (\cite[Section II.9]{BerMe}; see also \cite{HPS}) that
$$
\gamma(\R^d)<1 
\qquad\text{for all}\ d> 1\,,
$$
so \eqref{eq:pleijel} indeed improves \eqref{eq:courant} asymptotically. We also emphasize that the upper bound $\gamma(\R^d)$ in \eqref{eq:pleijel} only depends on $d$ and not on $\Omega$.

While the constant $\gamma(\R^d)$ appears not only in \eqref{eq:pleijel}, but also in various generalizations, including in the ones in this paper, it should be stressed that, at least when $d=2$, this constant is not the optimal constant in \eqref{eq:pleijel}. Improvements appear in \cite{Bou,Stei}. A candidate for the best constant is suggested in \cite{Po}.

B\'erard and Meyer \cite{BerMe} have shown the analogues of the Courant and Pleijel theorems for the Laplace--Beltrami operator on closed Riemannian manifolds. Interestingly, in their Pleijel-type bound the \emph{same} constant $\gamma(\R^d)$ as in the Euclidean case appears.

Returning to the case of an open subset of $\R^d$, but now with Neumann boundary conditions, already Courant proved the bound \eqref{eq:courant}. However, the analogue of Pleijel's bound \eqref{eq:pleijel} appeared only much later \cite{Po,Le,DPFaVi}. Again, the same constant $\gamma(\R^d)$ appears in these papers.

In passing we mention the recent works \cite{EL,FH}, where the Courant and Pleijel theorems were extended to the setting of sub-Riemannian manifolds. In this case the constant $\gamma(\R^d)$ is replaced by a constant depending on the Carnot groups appearing in the nilpotentization procedure and it is an open problem to verify that the resulting constant is strictly smaller than one for sufficiently large class of Carnot groups; see also \cite{Qiu}.

In the present paper we consider degenerate elliptic second order operators on open subsets of $\R^d$ and we prove the analogues of Courant's and Pleijel's theorems, the latter with the same constant $\gamma(\R^d)$ as in \eqref{eq:pl}. Instead of developing a general theory, we consider two typical model problems. The first model is given by the Baouendi--Grushin operator where a degeneracy occurs on a plane of arbitrary codimension, and the second model involves a degeneracy on the full boundary (so of codimension one). In both cases, we impose Dirichlet boundary conditions. These results appear as Theorems \ref{main} and \ref{maindeg}. We also sketch several possible generalizations of our result.

An important restriction of our result, which we do not know how to remove, concerns the speed of degeneration of ellipticity. We have to assume that the degeneration is not too strong in the sense that a certain power of the distance to the singular set needs to be integrable. This assumption comes through the use of Weyl asymptotics, as we explain momentarily. Whether our results remain valid without this assumption (and with the constant $\gamma(\R^d)$) remains an open problem.

As the definition of the Pleijel constant $\gamma(\R^d)$ in \eqref{eq:pl} indicates, the two ingredients in the proof of \eqref{eq:pleijel} are the Weyl asymptotic formula for large eigenvalues and the Faber--Krahn inequality. These two ingredients need to be refined in the various generalizations mentioned above. 

There is an extensive literature concerning Weyl asymptotics in the degenerate elliptic case and we refer to \cite{SV,BiSo,Lev} and references therein. The regime in which our proof of the Pleijel bound works is what is sometimes called, following Solomeshch \cite{Solo}, the case of `weak degeneracy', which is where the measure appearing in the Weyl formula is absolutely continuous with respect to Lebesgue measure. The Weyl formula is known in the complementary case as well, but then the Weyl measure is concentrated on the singular set where the ellipticity breaks down.

The second ingredient in Pleijel's proof is the Faber--Krahn inequality. A crucial observation of B\'erard and Meyer \cite{BerMe} was that one only needs a Faber--Krahn inequality for sets of small measure and, moreover, that it is acceptable to lose an arbitrarily small $\epsilon>0$ in the constant (with the upper bound on the measure of the set depending on $\epsilon$). A related idea was used by L\'ena in \cite{Le} to discard nodal sets close to the boundary in the Neumann case. In our setting we follow a similar approach, disregarding nodal sets close to the singular set where the ellipticity breaks down. The measure with respect to which we quantify the smallness of nodal sets is the one given by the Weyl asymptotics. As we have already mentioned, this measure is absolutely continuous with respect to Lebesgue measure, but the density is, in general, singular on the singular set. We can deal with this difficulty via weighted Sobolev inequalities with a precise relation between the integrability exponents and the exponents measuring the degeneracy of the weights. These bounds constitute one of the main technical novelties of this paper.

\subsection*{Acknowledgements}
We thank G.~Rozenblum for references concerning the Weyl law for degenerate elliptic operators and N.~Garofalo for discussions concerning unique continuation.


\section{Main result}\label{s1}

In this section we describe our main results in two different models of degenerate elliptic operators. Generalizations are discussed later in Section \ref{sec:final}.

\subsection{The first model: The Baouendi--Grushin operator}\label{s1a}~\\
We fix integers $n,m\geq 1$ and a parameter $\alpha\geq 0$. We write a point in $\R^{n+m}$ as $x=(x_1,x_2)$ with $x_1\in\R^n$ and $x_2\in\R^m$. Also, $\nabla_1$ and $\nabla_2$ denote the gradients with respect to these variables, and similarly $\Delta_1$ and $\Delta_2$ for the Laplacians.

We are interested in the Baouendi--Grushin operator\footnote{We refer to \cite{BGM} for the choice of this name.}
$$
-\Delta_1 - |x_1|^{2\alpha} \Delta_2
$$
or, more precisely, in its Dirichlet realization in $L^2(\Omega)$ for an open set $\Omega\subset\R^{n+m}$. We would like to prove Courant- and Pleijel-type theorems on the number of nodal domains of eigenfunctions.

It is known that eigenfunctions of this operator are continuous; see the discussion in the proof of Theorem \ref{main} below. Consequently, the nodal domains of an eigenfunction $u$ are well-defined as the connected components of the open set $\{x\in\Omega:\ u(x)\neq 0\}$. We let $\nu_k$ denote the supremum of the number of nodal domains over all eigenfunctions corresponding to the $k$-th eigenvalue (counting multiplicities).

\begin{theorem}\label{maincourant}
	Let $n\geq 2$ and let $\Omega\subset\R^{n+m}$ be a connected, open set such that the Dirichlet realization of the operator $-\Delta_1 - |x_1|^{2\alpha} \Delta_2$ in $L^2(\Omega)$ has discrete spectrum. Then
	$$
	\frac{\nu_k}{k} \leq 1
	\qquad\text{for all}\ k \,.
	$$
\end{theorem}

We do not know whether the assertion remains true for $n=1$. In this case we show $\nu_k\leq k$ for $k=1,2$ and $\nu_k\leq k+1$ for $k\geq 3$.

Our next theorem is an asymptotic improvement of Theorem \ref{maincourant}. We recall that the constant $\gamma(\R^d)$ is defined in \eqref{eq:pl}.

\begin{theorem}\label{main}
	Let $\Omega\subset\R^{n+m}$ be an open set satisfying
	\begin{equation}\label{hypW}
	\int_\Omega \frac{dx}{|x_1|^{\alpha m}} <\infty \,.
	\end{equation}
	Then
	$$
	\limsup_{k\to\infty} \frac{\nu_k}{k} \leq \gamma(\R^{n+m}) \,.
	$$
\end{theorem}

For $\alpha=0$, this is the standard Pleijel theorem. It is remarkable that the asymptotic upper bound on $\nu_k/k$ is independent of $\alpha$ if \eqref{hypW} holds.

We note that if $\Omega$ is bounded and $\Omega\cap\{x_1=0\}\neq\emptyset$, then \eqref{hypW}  is satisfied if and only if $\alpha < n/m$. We do now know whether a similar theorem holds for $\alpha\geq n/m$.


\subsection{The second model: Degeneration on the boundary}\label{s4}

\begin{assumption}\label{ass}
	Let $\Omega\subset\R^d$, $d\geq 2$, be a bounded open set with $C^1$-boundary. Let $0\leq \alpha<1$ and $\beta>-2$ with $\alpha -\beta<2/d$. Let $a\in C(\overline\Omega,\R^{d\times d})$ be Hermitian-valued and uniformly elliptic and let $b\in C(\overline\Omega)$ be positive.
\end{assumption}

Let
$$
d_\Omega(x):= \dist(x,\Omega^c) \,.
$$
Under the above assumption, we can define the operator
\begin{equation}\label{hypdeg}
	H := - d_\Omega^{-\beta} \,b^{-1} \,\nabla\cdot d_\Omega^\alpha \,a\, \nabla\,,
\end{equation}
with Dirichlet boundary conditions as a nonnegative selfadjoint operator in $L^2(\Omega,b \, d_\Omega^\beta)$. More precisely, we consider the quadratic form
\begin{equation}
	\label{eq:quadform}
u\mapsto 	\int_\Omega d_\Omega^\alpha \, \nabla u\cdot a \nabla u\,dx \,,
\end{equation}
initially defined on $C^1_c(\Omega)$. It is closable in the Hilbert space $L^2(\Omega,d_\Omega^\beta \, b)$ with norm
\begin{equation}
	\label{eq:norm}
u \mapsto \left( \int_\Omega d_\Omega^\beta\, b \,|u|^2\,dx \right)^{1/2},
\end{equation}
and  $H$ is the nonnegative operator corresponding to the closure of this quadratic form.

It is known that under Assumption \ref{ass} the operator $H$ has discrete spectrum; see the references for Lemma \ref{weyldeg} below.

Since the coefficients of the operator $H$ are continuous in $\Omega$, it is well-known that any eigenfunction is continuous in $\Omega$ (in fact, it belongs to $C^{0,\sigma}_{\rm loc}(\Omega)$ for any $\sigma<1$; this can be proved along the lines of \cite[Theorem 5.17]{GiMa}). Consequently, the nodal domains of an eigenfunction $u$ are well-defined as the connected components of the open set $\{x\in\Omega:\ u(x)\neq 0\}$. The quantity $\nu_k$ is defined as before.

The following theorem can be considered as known.

\begin{theorem}
	In addition to Assumption \ref{ass}, let $a$ be locally Lipschitz continuous in $\Omega$. Then
	$$
	\frac{\nu_k}{k} \leq 1 \qquad\text{for all}\ k \,.
	$$ 
\end{theorem}

Indeed, a closely related result is mentioned in the introduction of \cite{Al} in the uniformly elliptic case. The loss of ellipticity on the boundary, however, does not affect the proof. The Lipschitz continuity is needed for the unique continuation theorem of \cite{AKS}. (Note that if $a$ is Lipschitz, then $d_\Omega^\alpha a$ is Lipschitz as well away from $\partial\Omega = \{d_\Omega = 0\}$.)

Our new result in the present setting is the following Pleijel-type bound.

\begin{theorem}\label{maindeg}
	Under Assumption \eqref{ass}, we have
	$$
	\limsup_{k\to\infty} \frac{\nu_k}{k} \leq \gamma(\R^d) \,.
	$$
\end{theorem}

We recall that $\gamma(\R^d)$ is the constant from Pleijel's proof of the Laplacian. We find it remarkable that the same constant works for a rather large class of operators.


\subsection{Outline of the paper}

The remainder of this paper is structured as follows. In Section \ref{s2} we explain the strategy of the proof of the Pleijel bounds in Theorems \ref{main} and \ref{maindeg} by reducing them to the proof of Weyl asymptotics and certain Faber--Krahn-type inequalities.

Section \ref{sec:weyl} is devoted to the proof of Weyl asymptotics in the Baouendi--Grushin case, while in Sections \ref{s3} and \ref{s5} the Faber--Krahn-type bounds are derived. In Section~\ref{sec:final} we make some additional remarks on possible extensions of our results.

Section \ref{sec:courant} contains the proof of the Courant bound in Theorem \ref{maincourant}. 


\section{Proof strategy of the main results}\label{s2}

In this section we explain the strategy to prove the Pleijel-type bounds in Theorems \ref{main} and \ref{maindeg}. More precisely, we will reduce the proof of these results to the proof of certain Weyl-type asymptotics and Faber--Krahn-type inequalities, which will be proved in the remaining sections.

\subsection{Proof strategy in the Baouendi--Grushin case}\label{sec:strategybg}~\\

Throughout this section $\alpha\geq 0$, integers $n,m\geq 1$ are fixed and not reflected in the notation. It will be convenient to abbreviate
\begin{equation}\label{defmu}
\mu^{\rm BG}(\omega):= \int_\omega \frac{dx}{|x_1|^{\alpha m}} \,. 
\end{equation}

The first ingredient in our proof is a Weyl-type formula. Let $N(\lambda)$ denote the number of eigenvalues $<\lambda$, counting multiplicities, of the Dirichlet realization of $-\Delta_1 - |x_1|^{2\alpha}\Delta_2$ in $L^2(\Omega)$. We recall that the Weyl constant $\mathcal W(\R^d)$ for $d\geq 1$ is defined by~\eqref{eq:weyleucl}.

\begin{lemma}\label{LemmaW}
	Let $\Omega\subset\R^{n+m}$ be an open set with $\mu^{\rm BG}(\Omega)<\infty$. Then
	$$
	\lim_{\lambda\to\infty} \frac{N(\lambda)}{\lambda^{(n+m)/2}} = \mathcal W(\R^{n+m}) \int_\Omega \frac{dx}{|x_1|^{\alpha m}} \,.
	$$
\end{lemma}

There is a large literature on spectral asymptotics for degenerate elliptic operators, including \cite{Solo,BaGo,Met2,BCP,No,BiSo0,SV,SV0,Lev}. Since we have not been able to find the statement under our minimal assumptions, we provide a proof in Section \ref{sec:weyl} below. 

\medskip

To simplify the notation, we use the Grushin gradient
$$
\nabla^{(\alpha)} u = \begin{pmatrix}
	\nabla_1 u \\ |x_1|^\alpha \nabla_2 u
\end{pmatrix},
$$
so, in particular, $|\nabla^{(\alpha)}u|^2 = |\nabla_1 u|^2 + |x_1|^{2\alpha} |\nabla_2 u|^2$. For an open set $\Omega\subset\R^{n+m}$ we consider the space 
$$
S^{1,(\alpha)}(\Omega) := \{ u\in L^2(\Omega) :\ \nabla^{(\alpha)} u \in L^2(\Omega) \}
$$
(where $\nabla^{(\alpha)} u$ is understood in the sense of distributions) and we let $S^{1,(\alpha)}_0(\Omega)$ denote the closure of $C^1_c(\Omega)$ in $S^{1,(\alpha)}(\Omega)$.

The next ingredient in our proof is a Faber--Krahn-type inequality for open sets $\omega\subset\R^{n+m}$, involving their weighted measure $\mu^{\rm BG}(\omega)$ rather than their Euclidean measure $|\omega|$. We show that, if the former measure is small and the set is away from the critical set $\{x_1=0\}$, then the corresponding constant can be chosen arbitrarily close to the \emph{Euclidean Faber--Krahn constant} $C^{\rm FK}(\R^{n+m})$, defined by \eqref{eq:fkeucl}.

\begin{lemma}\label{almostfk}
	For every $\theta>0$ and $\delta>0$ there is an $\eta>0$ such that for every open set $\omega\subset\R^{n+m}$ with $\mu^{\rm BG}(\omega)\leq\eta$ and $\inf_{x\in\omega} |x_1|\geq \delta$ and every $u\in S^{1,(\alpha)}_0(\omega)$ one has
	$$
	\int_\omega |\nabla^{(\alpha)} u|^2\,dx \geq (1-\theta) \, C^{\rm FK}(\R^{n+m}) \, \mu^{\rm BG}(\omega)^{-\frac2{n+m}} \int_\omega |u|^2\,dx \,.
	$$
\end{lemma}

We will prove this lemma in Section \ref{s3}.

The final ingredient in our proof is a Faber--Krahn-type inequality that is valid for \emph{all} open sets $\omega$, without assuming that their $\mu^{\rm BG}$-measure is small or that they are away from the critical set $\{x_1=0\}$. The prize to be paid is that the constant is (probably) not sharp.

\begin{lemma}\label{fkunif}
	There is a constant $\widetilde C^{\rm FK}>0$ such that for every open set $\omega\subset\R^{n+m}$ with $\mu^{\rm BG}(\omega)<\infty$ and every $u\in S^{1,(\alpha)}_0(\omega)$ one has
	$$
	\int_\omega |\nabla^{(\alpha)} u|^2\,dx \geq \widetilde C^{\rm FK} \mu^{\rm BG}(\omega)^{-\frac2{n+m}} \int_\omega |u|^2\,dx \,.
	$$
\end{lemma}

\begin{remark}
	The constant $\widetilde C^{\rm FK}$ depends on $\alpha$, $n$ and $m$. For the best (that is, largest possible) constant, one can show that
	$$
	\widetilde C^{\rm FK} \leq C^{\rm FK}(\R^{n+m}) \,.
	$$
	Indeed, it suffices to take $\omega=B_r(a)$ where $|a_1|=1$, to let $r\to 0$ and to apply the standard Faber-Krahn estimate.
\end{remark}

With these ingredients at hand, we can now give the proof of our first main result. The proof combines ideas of \cite{BerMe} and \cite{Le}.

\begin{proof}[Proof of Theorem \ref{main}]
	Our proof will depend on several parameters that will all be chosen small at the end. We fix two parameters $\theta>0$ and $\delta>0$ and choose $\eta>0$ as in Lemma \ref{almostfk}. 
	Also, we pick two real-valued $C^1$ functions $\chi_0$ and $\chi_1$ on $\R^{n+m}$ with
	\begin{align*}
		& \chi_0^2 + \chi_1^2 =1 \,,\\
		& \chi_0 =1 \quad \text{in}\ \{ |x_1|\geq 2\delta \} \,,\\
		& \chi_1 = 1 \quad \text{in}\ \{|x_1|\leq \delta \} \,.
	\end{align*}
	
	Let $\lambda_k$ be an eigenvalue of the Dirichlet realization of $-\Delta_1 - |x_1|^{2\alpha} \Delta_2$ in $L^2(\Omega)$ and let $\psi$ be a corresponding normalized real eigenfunction. We know that $u$ is continuous in $\Omega$. Indeed, the function $\exp (\sqrt{\lambda_k}\, t) \psi$ satisfies $(-\Delta_1 - |x_1|^\alpha \Delta_2- \partial_t^2) \exp (\sqrt{\lambda_k}\, t) \psi= 0$ in $\Omega\times\R$ and is therefore locally H\"older continuous by \cite{FrLa}. (Concerning the assumption of $\lambda$-connectedness introduced there, we refer to the paragraph before \cite[Definition 2.4]{FrLa2}.) Therefore its nodal domains $\Omega_j$ are well defined as the connected components of $\{\psi\neq 0\}$. Let $\nu(\psi)$ denote the number of nodal domains. We will see soon that this number is finite.
	
	Let $\epsilon>0$ be another parameter. Following \cite{Le} we split
	$$
	\nu(\psi) = \nu^{(0)}_\leq(\psi) + \nu^{(0)}_>(\psi) + \nu^{(1)}(\psi) \,,
	$$
	where $\nu^{(0)}_\leq(\psi)$ denotes the number of nodal domains $\Omega_j$ with $\mu^{\rm BG}(\Omega_j)\leq\eta$ and
	$$
	\int_{\Omega_j} \chi_0^2 \, |\psi|^2\,dx \geq (1-\epsilon) \int_{\Omega_j} |\psi|^2\,dx \,,
	$$
	where $\nu^{(0)}_>(\psi)$ denotes the number of nodal domains $\Omega_j$ with $\mu^{\rm BG}(\Omega_j)>\eta$ and
	$$
	\int_{\Omega_j} \chi_0^2 \, |\psi|^2\,dx \geq (1-\epsilon) \int_{\Omega_j} |\psi|^2\,dx \,,
	$$
	and where $\nu^{(1)}(\psi)$ denotes the number of nodal domains $\Omega_j$ with
	$$
	\int_{\Omega_j} \chi_0^2 \, |\psi|^2\,dx  < (1-\epsilon) \int_{\Omega_j} |\psi|^2\,dx \,.
	$$
	
	For a nodal domain $\Omega_j$ of the type counted in $\nu^{(0)}_\leq$, we apply Lemma \ref{almostfk} to the function $u=\chi_0 \psi$ and the set $\omega = \Omega_j\cap\{|x_1|>\delta\}$. By the same argument as in \cite{FH} we have $u\in S^{1,(\alpha)}_0(\omega)$ and we obtain
	\begin{align*}
		& (1-\epsilon)(1-\theta) C^{\rm FK} \mu^{\rm BG}(\Omega_j)^{-\frac2{1+m}} \int_{\Omega_j} |\psi|^2\,dx \\
		& \quad \leq (1-\theta) C^{\rm FK} \mu^{\rm BG}(\Omega_j)^{-\frac2{1+m}}\int_{\Omega_j} \chi_0^2\, |\psi|^2\,dx \\
		& \quad \leq \int_{\Omega_j} |\nabla^{(\alpha)}(\chi_0\psi)|^2\,dx \\
		& \quad \leq (1+\epsilon') \int_{\Omega_j} \chi_0^2\, |\nabla^{(\alpha)}\psi|^2\,dx + (1+(\epsilon')^{-1}) \int_{\Omega_j} |\nabla\chi_0|^2 \, |\psi|^2\,dx \\
		& \quad \leq \Big( (1+\epsilon')\lambda_k + (1+(\epsilon')^{-1})\|\nabla\chi_0\|_\infty^2 \Big) \int_{\Omega_j} |\psi|^2\,dx \,.
	\end{align*}
	Here $\epsilon'>0$ is an arbitrary parameter and we have abbreviated $C^{\rm FK}=C^{\rm FK}(\R^{n+m})$. Thus,
	$$
	(1-\epsilon)^\frac{n+m}2 (1-\theta)^\frac{n+m}2 \left( C^{\rm FK} \right)^\frac{n+m}{2} \leq \left( (1+\epsilon')\lambda_k + (1+(\epsilon')^{-1})\|\nabla\chi_0\|_\infty^2 \right)^\frac{n+m}2 \mu^{\rm BG}(\Omega_j) \,.
	$$
	Summing with respect to $j$ gives
	\begin{equation}\label{in1}
		\begin{split}
	& (1-\epsilon)^\frac{n+m}2 (1-\theta)^\frac{n+m}2 \left( C^{\rm FK} \right)^\frac{n+m}{2} \nu^{(0)}_\leq(\psi) \\
	& \quad \leq \left( (1+\epsilon')\lambda_k + (1+(\epsilon')^{-1})\|\nabla\chi_0\|_\infty^2 \right)^\frac{n+m}2 \mu^{\rm BG}(\Omega) \,.
	\end{split}
	\end{equation}
	Next, we simply bound
	\begin{equation}\label{in2}
	\nu^{(0)}_>(\psi) \leq \eta^{-1} \mu^{\rm BG}(\Omega) \,.
	\end{equation}
	Finally, consider a nodal domain $\Omega_j$ of the type counted in $\nu^{(1)}(\psi)$ and note that we have
	$$
	\int_{\Omega_j} \chi_1^2\,  |\psi|^2\,dx > \epsilon \int_{\Omega_j} |\psi|^2\,dx \,.
	$$
	We apply Lemma \ref{fkunif} to the function $u=\chi_1 \psi$ and the set $\omega = \Omega_j\cap\{ |x_1|<2\delta\}$ and obtain
	\begin{align*}
		& \epsilon \widetilde C^{\rm FK} \mu^{\rm BG}(\Omega_j\cap\{|x_1|<2\delta\})^{-\frac2{n+m}} \int_{\Omega_j} |\psi|^2\,dx \\
		& \quad \leq \widetilde C^{\rm FK} \mu^{\rm BG}(\Omega_j\cap\{|x_1|<2\delta\})^{-\frac2{n+m}} \int_{\Omega_j} \chi_1^2\, |\psi|^2\,dx \\
		& \quad \leq \int_{\Omega_j} |\nabla^{(\alpha)}(\chi_1\psi)|^2\,dx \\
		& \quad \leq 2 \int_{\Omega_j} \chi_1^2\, |\nabla^{(\alpha)}\psi|^2\,dx + 2 \int_{\Omega_j} |\nabla\chi_1|^2 \, |u|^2\,dx \\
		& \quad \leq 2 \left( \lambda_k + \|\nabla\chi_1\|_\infty^2 \right) \int_{\Omega_j}\, |\psi|^2\,dx \,.
	\end{align*}
	Thus,
	$$
	\epsilon^\frac{n+m}{2} \left( \widetilde C^{\rm FK} \right)^\frac{n+m}{2} \leq 2^\frac{n+m}2 \left( \lambda_k + \|\nabla\chi_1\|_\infty^2 \right)^\frac{n+m}2 \mu^{\rm BG}(\Omega_j\cap\{|x_1|<2\delta\}) \,.
	$$
	Summing with respect to $j$ gives
	\begin{equation}\label{in3}
	\epsilon^\frac{n+m}{2} \left( \widetilde C^{\rm FK} \right)^\frac{n+m}{2} \nu^{(1)}(\psi) \leq 2^\frac{n+m}2 \left( \lambda_k + \|\nabla\chi_1\|_\infty^2 \right)^\frac{n+m}2 \mu^{\rm BG}(\Omega\cap\{|x_1|<2\delta\}) \,.
	\end{equation}
	
	Combining all these bounds, we see that $\nu(\psi)$ is finite and bounded by
	\begin{align*}
		\nu(\psi) & \leq (1-\epsilon)^{-\frac{n+m}2} (1-\theta)^{-\frac{n+m}2} \left( C^{\rm FK} \right)^{-\frac{n+m}{2}} \\
		& \quad\quad \times \left( (1+\epsilon')\lambda_k + (1+(\epsilon')^{-1})\|\nabla\chi_0\|_\infty^2 \right)^\frac{n+m}2 \mu^{\rm BG}(\Omega) \\
		& \quad + \eta^{-1} \mu^{\rm BG}(\Omega) \\
		& \quad + \epsilon^{-\frac{n+m}{2}} \left( \widetilde C^{\rm FK} \right)^{-\frac{n+m}{2}} 2^\frac{n+m}2 \left( \lambda_k + \|\nabla\chi_1\|_\infty^2 \right)^\frac{n+m}2 \mu^{\rm BG}(\Omega\cap\{|x_1|<2\delta\}) \,.
	\end{align*}
	We can take the supremum over all eigenfunctions $\psi$ corresponding to $\lambda_k$, divide by $k$ and take the limit $k\to\infty$ to obtain, in view of Weyl asymptotics (see Lemma \ref{LemmaW}),
	\begin{align*}
		\limsup_{k\to\infty} \frac{\nu_k}{k} & \leq \left( \mathcal W(\R^{n+m}) \right)^{-1} (1-\epsilon)^{-\frac{n+m}2} (1-\theta)^{-\frac{n+m}2} \left( C^{\rm FK} \right)^{-\frac{n+m}{2}} (1+\epsilon')^\frac{n+m}2 \mu^{\rm BG}(\Omega) \\
		& \quad + \left( \mathcal W(\R^{n+m}) \right)^{-1} \epsilon^{-\frac{n+m}{2}} \left( \widetilde C^{\rm FK} \right)^{-\frac{n+m}{2}} 2^\frac{n+m}2 \frac{\mu^{\rm BG}(\Omega\cap\{|x_1|<2\delta\})}{\mu^{\rm BG}(\Omega)} \,.
	\end{align*}
	Letting first $\delta\to 0$, $\theta\to 0$ and $\epsilon'\to 0$ and then $\epsilon\to 0$, we obtain, using \eqref{eq:pl}, the bound claimed in the theorem.
\end{proof}


\subsection{Proof strategy in the case of boundary degeneration}\label{sec:strategydeg}~\\

We explain the proof of Theorem \ref{maindeg}, focusing, in particular, on the differences from that of Theorem \ref{main}. The first ingredient is, as usual, a Weyl formula. Let $N(\lambda)$ denote the number of eigenvalues $<\lambda$, counting multiplicities, of the operator $H$ introduced in Subsection~\ref{s4}.

\begin{lemma}\label{weyldeg}
	Under Assumption \eqref{ass}, we have
	\begin{equation}
		\lim_{\lambda\to\infty} \frac{N(\lambda)}{\lambda^{d/2}} = \mathcal W(\R^d) \int_\Omega d_\Omega^{-(\alpha-\beta)d/2} \frac{b^{d/2}}{\sqrt{\det a}}\,dx \,.
	\end{equation}
\end{lemma}

Under the additional assumption $\alpha + \max\{-\beta,0\}<2/d$, this proposition is a consequence of the results of Birman and Solomyak \cite{BiSo}. The general case is due to Tashchiyan \cite{Ta}. We also refer to \cite[Appendix 8]{BiSo0} for a review of these results and references to earlier partial results.

The other ingredients of the proof of Theorem \ref{maindeg} are Faber--Krahn-type inequalities. The natural Sobolev space associated to this problem is
$$
H^{1,(\alpha,\beta)}(\Omega):= \{ u\in L^2(\Omega,d_\Omega^\beta\,dx) :\ \nabla u \in L^2(\Omega,d_\Omega^\alpha\,dx) \} \,,
$$
where the gradient is understood in the sense of distributions. Also, for an open set $\omega\subset\Omega$, let $H^{1,(\alpha,\beta)}_0(\omega)$ denote the closure of $C^1_c(\omega)$ in $H^{1,(\alpha,\beta)}(\Omega)$. Note that $H^{1,(\alpha,\beta)}_0(\omega)$ depends not only on $\omega$ but also on $\Omega$ through the choice of the weights.

Let us set, for measurable $\omega\subset\Omega$,
\begin{equation}
	\mu(\omega):= \int_\omega d_\Omega^{-(\alpha-\beta)d/2} \frac{b^{d/2}}{\sqrt{\det a}}\,dx \,.
\end{equation}
The following is an `almost-Euclidean' Faber--Krahn-type inequality in the spirit of B\'erard and Meyer \cite{BerMe}.

\begin{lemma}\label{almostfkdeg}
	Under the above assumptions, for every $\theta>0$ and $\delta>0$ there is an $\eta>0$ such that for every open set $\omega\subset\Omega$ with $\mu(\omega)\leq\eta$ and $\inf_{x\in\omega} d_\Omega(x)\geq \delta$ and every $u\in H^{1,(\alpha,\beta)}_0(\omega)$ one has
	\begin{equation}
		\int_\omega d_\Omega^\alpha\, \nabla u \cdot a \nabla u\,dx \geq (1-\theta) \, C^{\rm FK}(\R^d) \, \mu(\omega)^{-\frac{2}{d}} \int_\omega d_\Omega^\beta\, b\, |u|^2\,dx \,.
	\end{equation}
\end{lemma}

The final ingredient in the proof of Theorem \ref{maindeg} is a Faber--Krahn-type inequality with a non-explicit constant, but which, in contrast to the previous lemma, is valid for all open subsets.

\begin{lemma}\label{fkunifdeg}
	Let $\Omega\subset\R^d$ be a bounded open set with $C^1$-boundary and let $0\leq\alpha<1$ and $\beta>-2$ with $\alpha - \beta <2/d$. Then there is a constant $c>0$ such that for every open set $\omega\subset\Omega$ and every $u\in H^{1,(\alpha,\beta)}_0(\omega)$ one has
	\begin{equation}
		\int_\omega d_\Omega^\alpha\, |\nabla u|^2 \,dx \geq c \left( \int_\omega d_\Omega^{-(\alpha -\beta)d/2}\,dx \right)^{-\frac{2}{d}} \int_\omega d_\Omega^\beta \, |u|^2\,dx \,.
	\end{equation}
\end{lemma}

Given these three lemmas, the proof of Theorem \ref{maindeg} is essentially the same as in the Baouendi--Grushin case. We omit it.


\section{Weyl asymptotics}\label{sec:weyl}

In this section we prove the Weyl asymptotics of Lemma \ref{LemmaW}. We emphasize, in particular, that they are valid under the sole assumption that $\mu^{\rm BG}(\Omega)<\infty$, where $\mu^{\rm BG}$ is defined in \eqref{defmu}.

Our proof is based on a method developed recently in \cite{CaFrLa}, based on earlier work in \cite{Si}. We comment on another method at the end of the proof.

\begin{proof}[Proof of Lemma \ref{LemmaW}]~\\
	Let $H_\Omega$ denote the Dirichlet realization of the operator $-\Delta_1 - |x_1|^{2\alpha} \Delta_2$ in $L^2(\Omega)$. By the Tauberian theorem (see, e.g., \cite[Theorem 10.3]{SiFI}), it suffices to prove
	$$
	\lim_{t\to 0} t^\frac{n+m}{2}\Tr\exp(-tH_\Omega) = (4\pi)^{-\frac{n+m}{2}} \mu^{\rm BG}(\Omega) \,.
	$$
	
	For $x_1\in\R^n$, we write $\Omega(x_1) := \{ x_2\in\R^m:\ (x_1,x_2)\in\Omega\}$. Then, by applying the Golden--Thompson inequality as in \cite{Si}, we obtain, for any $t>0$,
	$$
	\Tr_{L^2(\Omega)} \exp(-tH_\Omega) \leq (4\pi t)^{-\frac{n}{2}} \int_{\R^n} \Tr_{L^2(\Omega(x_1))} \exp(t|x_1|^{2\alpha}\Delta_{\Omega(x_1)}) \,dx_1 \,.
	$$ 
	Here $-\Delta_{\Omega(x_1)}$ denotes the Dirichlet Laplacian in $L^2(\Omega(x_1))$. It is well known that, for any open set $\omega\subset\R^m$ of finite measure and any $s>0$,
	$$
	\Tr_{L^2(\omega)} \exp(s\Delta_\omega) \leq (4\pi)^{-\frac m2}|\omega| \,.
	$$
	Applying this with $\omega = \Omega(x_1)$ and $s=t|x_1|^{2\alpha}$, we arrive at
	\begin{equation}
		\label{eq:kac}
		\Tr\exp(-tH_\Omega) \leq (4\pi t)^{-\frac{n}{2}} \int_{\R^n} (4\pi t|x_1|^{2\alpha})^{-\frac m2}|\Omega(x_1)| \,dx_1 = (4\pi t)^{-\frac{n+m}{2}} \mu^{\rm BG}(\Omega) \,.
	\end{equation}
	This is the claimed upper bound, even in non asymptotic form.\\
	
	For the asymptotic lower bound, we fix a symmetric decreasing function $g \in C^1_c(\R^n)$ with $\|g\|_{L^2(\R^n)} =1$ and set, for $(x_1,\xi_1)\in \R^n\times\R^n$,
	$$
	\psi_{x_1,\xi_1}(y_1) := e^{i\xi_1\cdot x_1} g(y_1-x_1) 
	\qquad\text{for}\ y_1\in\R^n \,.
	$$
	Then, by the Plancherel theorem,
	$$
	\iint_{\R^n\times\R^n} |\psi_{x_1,\xi_1}\rangle\langle\psi_{x_1,\xi_1}| \,\frac{dx_1\,d\xi_1}{(2\pi)^n} = \1_{L^2(\R^n)}\,,
	$$
	and therefore
	$$
	\Tr \exp(-tH_\Omega) = \iint_{\R^n\times\R^n} \Tr_{L^2(\R^{n+m})} \left( \left( |\psi_{x_1,\xi_1}\rangle\langle\psi_{x_1,\xi_1}| \otimes \1_{L^2(\R^m)} \right) \exp(-tH_\Omega) \right) \frac{dx_1\,d\xi_1}{(2\pi)^n} \,.
	$$
	Here we consider $H_\Omega$ as a (not densely defined) operator in $L^2(\R^{n+m})$. As shown in \cite{CaFrLa}, one has, for each $(x_1,\xi_1)\in\R^n\times\R^n$,
	$$
	\Tr_{L^2(\R^{n+m})} \left( |\psi_{x_1,\xi_1}\rangle\langle\psi_{x_1,\xi_1}| \exp(-tH_\Omega) \right) \geq \Tr_{L^2(\R^m)} \exp(-t H(x_1,\xi_1)) \,,
	$$
	where $H(x_1,\xi_1)$ is the not densely defined operator in $L^2(\R^m)$ associated to the quadratic form whose form domain consists of all functions $u\in L^2(\R^m)$ such that $\psi_{x_1,\xi_1}\otimes u$ belongs to the form domain of $H_\Omega$, and
	\begin{align*}
		(u,H(x_1,\xi_1)u) & = ( \psi_{x_1,\xi_1}\otimes u, H_\Omega(\psi_{x_1,\xi_1}\times u)) \\
		& = (|\xi_1|^2 + \|\nabla g\|_2^2) \|u\|_2^2 + |x_1|^{2\alpha} \int_{\R^m} |\nabla_2 u|^2\,dx_2 \,.
	\end{align*}
	Fix $\epsilon>0$ and set $\Omega_\epsilon:=\{ x \in\R^{n+m} :\ \dist(x,\Omega^c)>\epsilon\}$. Also, let $g(x_1)=\epsilon^{-n/2} G(x_1/\epsilon)$ for a symmetric decreasing function $G \in C^1_c(\R^n)$ with $\|G\|_{L^2(\R^n)} =1$ and $\supp G \subset B_1$ (where $B_1$ is the unit ball). Then we see that any $u\in H^1_0(\Omega_\epsilon(x_1))$ belongs to the form domain of $H(x_1,\xi_1)$. It follows that, in the sense of quadratic forms,
	$$
	H(x_1,\xi_1) \leq |\xi_1|^2 + \|\nabla g\|_2^2 - |x_1|^{2\alpha} \Delta_{\Omega_\epsilon(x_1)} \,,
	$$
	and therefore
	$$
	\Tr_{L^2(\R^m)} \exp(-t H(x_1,\xi_1)) \geq e^{-t(|\xi_1|^2+\|\nabla g\|_2^2)} \Tr_{L^2(\R^m)} \exp(t|x_1|^{2\alpha} \Delta_{\Omega_\epsilon(x_1)}) \,.
	$$
	Computing the $\xi$-integral, we arrive at the lower bound
	$$
	\Tr\exp(-tH_\Omega) \geq (4\pi t)^{-\frac{n}{2}} e^{-t\|\nabla g\|_2^2} \int_{\R^n} \Tr_{L^2(\R^m)} \exp(t|x_1|^{2\alpha} \Delta_{\Omega_\epsilon(x_1)}) \,dx_1 \,.
	$$
	It is well-known that, for any open set $\omega\subset\R^m$ of finite measure and any $s>0$, we have
	$$
	\lim_{s\to 0} s^\frac m2 \Tr\exp(s\Delta_\omega) = (4\pi)^{-\frac m2}|\omega| \,.
	$$
	Thus, by Fatou's lemma, we obtain
	$$
	\liminf_{t\to 0} t^\frac{n+m}{2} \Tr\exp(-tH_\Omega) \geq (4\pi)^{-\frac{n}{2}} \int_{\R^n} |x_1|^{-m\alpha} \, |\Omega_\epsilon(x_1)| \,dx_1 = (4\pi)^{-\frac{n+m}{2}} \mu^{\rm BG}(\Omega_\epsilon) \,.
	$$
	Finally, we can let $\epsilon\to 0$ and use monotone convergence to get the claimed asymptotic lower bound.
\end{proof}

The previous proof has the advantage of being rather short and not requiring any major results (besides the Golden--Thompson inequality and Jensen's inequality). On the other hand, it relies on the exact product structure of the Baouendi--Grushin operator. We now sketch a more robust proof of Lemma \ref{LemmaW} that, however, relies on results by Birman and Solomyak.

\begin{proof}[Alternative proof of Lemma \ref{LemmaW}]
	Let $N(\lambda,A)$ denote the number of eigenvalues $\leq\lambda$, counting multiplicities, of a lower semibounded operator $A$.
	
	Let $\epsilon, R >0$ and set
	$$
	\Omega_{\epsilon,R}:= \{ x\in\Omega :\ |x_1| >\epsilon \,,\ |x|<R \} \,.
	$$
	By the variational principle, we have
	$$
	N(\lambda,H_\Omega) \geq N(\lambda, H_{\Omega_{\epsilon,R}}) \,.
	$$
	Since $H_{\Omega_{\epsilon,R}}$ is a uniformly elliptic operator on a bounded domain, we have
	\begin{equation}\label{eq:weylstandard}
		\lim_{\lambda\to\infty} \lambda^{-\frac{n+m}2} N(\lambda, H_{\Omega_{\epsilon,R}}) = \mathcal W(\R^{n+m}) \mu^{\rm BG}(\Omega_{\epsilon,R}) \,;
	\end{equation}
	see, for instance, \cite[Theorem 5.8]{BiSo0} for a much more general result. Thus,
	$$
	\liminf_{\lambda\to\infty} \lambda^{-\frac{n+m}2} N(\lambda, H_{\Omega}) \geq \mathcal W(\R^{n+m}) \mu^{\rm BG}(\Omega_{\epsilon,R}) \,.
	$$
	Since this is valid for any $\epsilon,R>0$, monotone convergence gives the claimed asymptotic lower bound.
	
	To prove the asymptotic upper bound, we fix again $\epsilon,R>0$ and choose $C^1$ functions $\chi_0$ and $\chi_1$ on $\R^{n+m}$ such that $\chi_0^2 + \chi_1^2 = 1$, $\supp\chi_0 \subset\{x\in\R^{n+m}:\ |x_1|>\epsilon \,, |x|<R\}$ and $\chi_0 =1$ in $\{ x\in\R^{n+m}:\ |x_1|>2\epsilon \,, |x|<R/2\}$.
	Then, by the IMS localization formula,
	$$
	H_\Omega = \chi_0 \left( H_\Omega - |\nabla\chi_0|^2 - |\nabla\chi_1|^2 \right) \chi_0 + \chi_1 \left( H_\Omega - |\nabla\chi_0|^2 - |\nabla\chi_1|^2 \right) \chi_1
	$$
	and therefore, by the variational principle,
	$$
	N(\lambda,H_\Omega) \geq N(\lambda, \chi_0 \left( H_\Omega - |\nabla\chi_0|^2 - |\nabla\chi_1|^2 \right) \chi_0) + N(\lambda, \chi_1 \left( H_\Omega - |\nabla\chi_0|^2 - |\nabla\chi_1|^2 \right) \chi_1) \,.
	$$
	Let $\Lambda:= \sup (|\nabla\chi_0|^2 + |\nabla\chi_1|^2)$. Then
	\begin{align*}
		& N(\lambda, \chi_0 \left( H_\Omega - |\nabla\chi_0|^2 - |\nabla\chi_1|^2 \right) \chi_0) \leq N(\lambda + \Lambda, H_{\Omega_{\epsilon,R}}) \,,\\
		& N(\lambda, \chi_1 \left( H_\Omega - |\nabla\chi_0|^2 - |\nabla\chi_1|^2 \right) \chi_1) \leq N(\lambda + \Lambda, H_{\Omega \setminus \overline{\Omega_{2\epsilon,R/2}}}) \,.
	\end{align*}
	Thus, in view of \eqref{eq:weylstandard}, we have
	$$
	\limsup_{\lambda\to\infty} \lambda^{-\frac{n+m}2} N(\lambda, H_{\Omega}) \leq \mathcal W(\R^{n+m}) \mu^{\rm BG}(\Omega_{\epsilon,R}) + \limsup_{\mu\to\infty}  \mu^{-\frac{n+m}2} N(\mu, H_{\Omega \setminus \overline{\Omega_{2\epsilon,R/2}}}) \,.
	$$
	By dominated convergence, the first term gives the claimed asymptotic bound as $\epsilon\to 0$ and $R\to\infty$ and it remains to prove that the second term vanishes in this limit. This is a consequence of the following lemma, applied with $\Omega$ replaced by $\Omega \setminus \overline{\Omega_{2\epsilon,R/2}}\,$, and using $\mu^{\rm BG}(\Omega\setminus \overline{\Omega_{2\epsilon,R/2}})\to 0$ as $\epsilon\to 0$ and $R\to\infty$ by dominated convergence. Thus, the proof of Lemma \ref{LemmaW} is complete, once the following lemma is proved.
\end{proof}

\begin{lemma}
	For any $n$ and $m$ there is a constant $K$ such that for any $\alpha\geq 0$, for any open set $\Omega\subset\R^{n+m}$ with $\mu^{\rm BG}(\Omega)<\infty$ and for any $\lambda\geq 0$, we have
	\begin{equation}\label{eq:K}
	N(\lambda,H_\Omega) \leq K \mu^{\rm BG}(\Omega) \lambda^\frac{n+m}{2} \,.
	\end{equation}
\end{lemma}

\begin{proof}
	It follows from \eqref{eq:kac}  that, for any $\lambda\geq 0$ and any $t\geq 0$, we have
	$$
	N(\lambda,H_\Omega) \leq e^{t\lambda} \Tr\exp(-tH_\Omega) \leq (4\pi t)^{-\frac{n+m}{2}} e^{t\lambda} \mu^{\rm BG}(\Omega) \,.
	$$
	Choosing $t = c/\lambda$ with a constant $c$ (the optimal choice is $c=\frac{n+m}{2}$), we arrive at the claimed bound.	
\end{proof}

\begin{remark}
	The above proof gives an explicit value for the constant $K$ in \eqref{eq:K}. If instead of working with the heat trace $\Tr\exp(-tH_\Omega)$ we work with the Riesz means $\Tr(H_\Omega - \lambda)^{3/2}_-$ and instead of using the Golden--Thompson inequality from \cite{Si} we use the operator-valued Lieb--Thirring inequality by Laptev and Weidl \cite{LaWe}, we obtain a smaller value for the constant $K$.
\end{remark}

\begin{remark}
	In the previous proof we used \eqref{eq:kac}. Alternatively, we could use the fact that, by the maximum principle, the heat kernel of $H_\Omega$ is pointwise bounded from above by the heat kernel of $H_{\R^{n+m}}$. Using the bound \eqref{eq:heatkernelbound} below, we arrive again at a bound $\Tr\exp(-tH_\Omega) \leq K t^{-\frac{n+m}{2}} \mu^{\rm BG}(\Omega)$, except that now the constant $K$ might depend on $\alpha$ (in addition to $n$ and $m$). It is still independent of $\Omega$, which is what is needed in the alternative proof of Lemma \ref{LemmaW}.
\end{remark}


\section{Almost-Euclidean Faber--Krahn-type inequalities}\label{s3}

Our goal in this section is to prove the almost-Euclidean Faber--Krahn inequalities in Lemmas \ref{almostfk} and \ref{almostfkdeg}. These inequalities say that, for sets of small measure and away from the singular set, the Faber--Krahn constant holds almost with the Euclidean constant.\\

\subsection{The almost-Euclidean Faber--Krahn inequality in the Baouendi--Grushin case}~\\
In this subsection we prove Lemma \ref{almostfk}.

We fix a real function $\chi\in C^1_c(\R^n)$ with $\supp\chi \subset \overline{B_1}$ and $\|\chi\|_2=1$, and set, with a parameter $r>0$ to be specified later,
$$
\chi_a(x_1):= r^{-\frac{n}{2}} \chi((x_1-a)/r)
\qquad\text{for all}\ x_1,a\in\R^n \,.
$$
Note that
\begin{equation}
	\label{eq:ims0}
	\int_{\R^n} \chi_a(x_1)^2 \,da = 1
	\qquad\text{for all}\ x_1\in\R^n \,,
\end{equation}
which implies
\begin{equation}
	\label{eq:ims0a}
	\int_{\R^n} \nabla_1 \chi_a(x_1) \;   \chi_a (x_1) \,da = 0
	\qquad\text{for all}\ x_1\in\R^n \,.
\end{equation}
By a continuous version of the IMS formula, we have for $u\in S^{1,(\alpha)}(\omega)$,
\begin{equation}
	\label{eq:ims}
	\int_{\omega} |\nabla^{(\alpha)}u|^2\,dx = \int_{\R^n} \int_\omega |\nabla^{(\alpha)}(\chi_a u)|^2 \,dx\,da - r^{-2}\|\nabla\chi\|_2^2 \int_\omega |u|^2\,dx \,.
\end{equation}

We fix $a\in\R^n$ with $|a|>r$ and bound
$$
\int_\omega |\nabla^{(\alpha)}(\chi_a u)|^2 \,dx \geq \int_\omega (|\nabla_1(\chi_a u)|^2 + (|a| - r)^{2\alpha} |\nabla_2(\chi_a u)|^2) \,dx \,.
$$
Let us write
$$
u(x) = \widetilde u(x_1, (|a|-r)^{-\alpha} x_2) \,.
$$
It follows that
$$
\int_\omega (|\nabla_1(\chi_a u)|^2 + (|a| - r)^{2\alpha} |\nabla_2(\chi_a u)|^2) \,dx
= (|a|-r)^{m\alpha} \int_{\widetilde\omega} (|\nabla_1(\chi_a \widetilde u)|^2 + |\nabla_2(\chi_a \widetilde u)|^2) \,dy\,,
$$
with
$$
\widetilde\omega := \{ y\in\R^{n+m} :\ (y_1, (|a|-r)^{\alpha} y_2)\in\omega \} \,.
$$
By the Euclidean Faber--Krahn inequality (abbreviating $C^{\rm FK}=C^{\rm FK}(\R^{n+m})$), we have
\begin{align*}
	\int_{\widetilde\omega} (|\nabla_1(\chi_a \widetilde u)|^2 + |\nabla_2(\chi_a \widetilde u)|^2) \,dy & \geq C^{\rm FK} |\widetilde\omega\cap \supp\chi_a |^{-\frac2{n+m}} \int_{\widetilde\omega} |\widetilde u|^2\,dy \\
	& = C^{\rm FK} |\widetilde\omega\cap \supp\chi_a |^{-\frac2{n+m}} (|a|-r)^{-m\alpha} \int_{\omega} | \chi_a u|^2\,dx \,.
\end{align*}
Thus, we have shown
$$
\int_\omega |\nabla^{(\alpha)}(\chi_a u)|^2 \,dx \geq C^{\rm FK} |\widetilde\omega\cap\supp\chi_a|^{-\frac2{n+m}} \int_{\omega} |\chi_a u|^2\,dx \,.
$$

It remains to estimate the measure on the right side. We have
\begin{align*}
	|\widetilde\omega\cap\supp\chi_a| & = (|a|-r)^{-m\alpha} \int_{\omega\cap\supp\chi_a} \,dx \leq \int_{\omega\cap\supp\chi_a} (|x_1|-2r)^{-m\alpha} \,dx \\
	& \leq \int_{\omega} (|x_1|-2r)^{-m\alpha} \,dx \,.
\end{align*}
Assuming $r<\delta/2$ we have, when $\inf_{x\in\omega}|x_1|\geq\delta$, that
$$
(|x_1|-2r)^{-1} \leq (\delta - 2r)^{-1} \delta |x_1|^{-1} \,,
$$
so
$$
|\widetilde\omega\cap\supp\chi_a| \leq (1- 2r\delta^{-1})^{-m\alpha} \mu(\omega) \,.
$$
Correspondingly, our inequality becomes
\begin{equation}
	\label{eq:almostfkproof}
	\int_\omega |\nabla^{(\alpha)}(\chi_a u)|^2 \,dx \geq (1- 2r\delta^{-1})^{\frac{2m\alpha}{n+m}} C^{\rm FK} \mu(\omega)^{-\frac2{n+m}} \int_{\omega} |\chi_a u|^2\,dx \,.
\end{equation}
We recall that we have proved this inequality under the assumptions $|a|>r$ and $r<\delta/2$.

Meanwhile, for $|a|\leq r$ (still assuming $r<\delta/2$) we have $\chi_a u = 0$, because $\supp\chi_a \subset \{ |x_1|\leq |a|+r\} \subset\{ |x_1|\leq 2r \}$ and $\inf_{x\in\omega} |x_1| \geq \delta$. Therefore, inequality \eqref{eq:almostfkproof} is trivially valid for $|a|\leq r$ as well. Integrating the inequality with respect to $a$ over $\R^n$ and recalling \eqref{eq:ims0} and \eqref{eq:ims}, we arrive at
$$
\int_\omega |\nabla^{(\alpha)} u|^2\,dx \geq \left( (1- 2r\delta^{-1})^{\frac{2m\alpha}{n+m}} c^{\rm FK} \mu(\omega)^{-\frac2{n+m}} - r^{-2} \|\nabla\chi\|_2^2 \right)
\int_{\omega} | u|^2\,dx \,.
$$
When $\mu(\omega)\leq\eta$, we obtain from this
$$
\int_\omega |\nabla^{(\alpha)} u|^2\,dx \geq \left( (1- 2r\delta^{-1})^{\frac{2m\alpha}{n+m}} - r^{-2} \eta^\frac{2}{n+m} (C^{\rm FK})^{-1} \|\nabla\chi\|_2^2 \right) C^{\rm FK} \mu(\omega)^{-\frac2{n+m}}
\int_{\omega} | u|^2\,dx \,.
$$
This is valid for all $r<\delta/2\,$.

Now given $\theta>0$, we first choose $r>0$ so small that
$$
(1- 2r\delta^{-1})^{\frac{2m\alpha}{n+m}} \geq 1- \theta/2 \,,
$$
and then we choose $\eta>0$ so small that
$$
r^{-2} \eta^\frac{2}{n+m} (C^{\rm FK})^{-1} \|\nabla\chi\|_2^2 \leq \theta/2 \,.
$$
This completes the proof of Lemma \ref{almostfk}.
\qed
~\\


\subsection{The almost-Euclidean Faber--Krahn inequality in the case of boundary degeneration}~\\
In this subsection we prove Lemma \ref{almostfkdeg}.

Similarly as in the proof of Lemma \ref{almostfk}, we fix a real function $\chi\in C^1_c(\R^d)$ with $\supp\chi\subset \overline{B_1}$ and $\|\chi\|_2=1$ and set, with a parameter $r>0$ to be specified later,
$$
\chi_z(x) := r^{-\frac d2} \chi((x-z)/r)
\qquad\text{for all}\ x,z\in\R^d \,.
$$

Note that
\begin{align*}
	& \frac12\left( \nabla (\chi_z^2 u)\cdot a \nabla u + \nabla u\cdot a\nabla (\chi_z^2 u) \right) = \nabla(\chi_z u)\cdot a \nabla(\chi_z u) - |u|^2 \nabla\chi_z\cdot a\nabla\chi_z \,.
\end{align*}
As a consequence we obtain the IMS-type formula
$$
\int_{\omega} d_\Omega^\alpha \nabla u \cdot a\nabla u\,dx = \int_{\R^d} \int_\omega d_\Omega^\alpha \,\nabla(\chi_z u) \cdot a\nabla(\chi_z u)\,dx \,dz - \int_\omega \Lambda\,  |u|^2\,dx
$$
where
$$
\Lambda(x):= r^{-2} d_\Omega(x)^\alpha \int_{\R^d} \nabla\chi(y) \cdot a(x) \nabla\chi(y)\,dy \,.
$$
Clearly, we have
$$
|\Lambda| \leq r^{-2} \,d_\Omega^\alpha\, \|a\|_{L^\infty(\Omega)} \, \|\nabla\chi\|_2^2 \,.
$$
Thus,
$$
\int_\omega \Lambda |u|^2\,dx \leq r^{-2} \sup_\omega d_\Omega^{\alpha-\beta} \|a\|_{L^\infty(\Omega)} (\inf_\Omega b)^{-1} \|\nabla\chi\|_2^2 \int_\omega d_\Omega^\beta \,b\, |u|^2\,dx \,.
$$
If $\alpha -\beta \geq 0$, then
$$
\sup_\omega d_\Omega^{\alpha-\beta} \leq \mathrm{diam}(\Omega)^{\alpha -\beta}<\infty \,.
$$
If $\alpha - \beta<0$ we use the assumption of Lemma \ref{almostfkdeg} that $\inf_\omega d_\Omega \geq \delta>0$ and deduce
$$
\sup_\omega d_\Omega^{\alpha-\beta} \leq \delta^{\alpha -\beta}<\infty \,.
$$
Thus, in either case we have, provided $\mu(\omega)\leq\eta\,$,
$$
\int_\omega \Lambda |u|^2\,dx \leq C_{r,\delta} \int_\omega d_\Omega^\beta\, b\, |u|^2\,dx 
\leq \eta^{\frac2d} C_{r,\delta}\, \mu(\omega)^{-\frac 2d}\int_\omega d_\Omega^\beta\, b\, |u|^2\,dx \,.
$$
This is the desired bound on the localization error.

Let us deal with the main term. Since $a$ is uniformly positive definite and uniformly continuous, for every $\epsilon>0$ there is a $\rho>0$ such that
$$
\xi\cdot a(x) \xi \geq (1-\epsilon) \xi\cdot a(y) \xi
\qquad\text{for all}\ \xi\in\R^d \,, x,y\in\Omega \ \text{with}\ |x-y|\leq\rho \,.
$$
Let us consider $z\in\R^d$ with $d_\Omega(z)> r$ and assume that $r\leq\rho$. Then
$$
\int_\omega d_\Omega^\alpha \nabla(\chi_z u) \cdot a\nabla(\chi_z u)\,dx
\geq (1-\epsilon)(d_\Omega(z)-r)^\alpha \int_\omega \nabla(\chi_z u) \cdot a(z)\nabla(\chi_z u)\,dx \,.
$$
We change coordinates such that the matrix $a(z)$ is diagonal, apply the Euclidean Faber--Krahn inequality (abbreviated $C^{\rm FK}=C^{\rm FK}(\R^d)$) and change back to the original coordinates. This gives
$$
\int_\omega \nabla(\chi_z u) \cdot a(z)\nabla(\chi_z u)\,dx \geq C^{\rm FK} |\omega\cap\supp\chi_z|^{-\frac2d} (\det a(z))^{\frac 1d} \int_\omega |\chi_z u|^2\,dx
$$
Note that
$$
\det a(z) \geq (1-\epsilon)^d \det a(x)
\qquad\text{for all}\ x\in\supp\chi_z \,,
$$
so
$$
|\omega\cap\supp\chi_z|^{-\frac2d} (\det a(z))^{\frac 1d}
\geq (1-\epsilon) \left( \int_{\omega\cap\supp\chi_z} \frac{dx}{\sqrt{\det a(x)}} \right)^{-\frac 2d}.
$$
Finally, since $b$ is positive and uniformly continuous, we may assume, after decreasing $\rho$ if necessary, that
$$
b(x) \geq (1-\epsilon) b(y)
\qquad\text{for all}\ x,y\in\Omega \ \text{with}\ |x-y|\leq\rho \,.
$$
It follows that
$$
\left( \int_{\omega\cap\supp\chi_z} \frac{dx}{\sqrt{\det a(x)}} \right)^{-\frac 2d} \int_\omega |\chi_z u|^2\,dx \geq (1-\epsilon) \left( \int_{\omega\cap\supp\chi_z} b(x)^\frac d2 \frac{dx}{\sqrt{\det a(x)}} \right)^{-\frac 2d} \int_\omega b |\chi_z u|^2\,dx
$$
Finally,
$$
\int_\omega b |\chi_z u|^2\,dx \geq (d_\Omega(z)+ (\sgn\beta) r)^{-\beta} \int_\omega d_\Omega^\beta b |\chi_z u|^2\,dx \,.
$$
We note that
$$
(d_\Omega(z)-r)^\alpha (d_\Omega(z)+ (\sgn\beta) r)^{-\beta}
\geq (d_\Omega(x)-2r)^\alpha (d_\Omega(x)+ 2(\sgn\beta) r)^{-\beta}
\qquad\text{for}\ x\in\supp\chi_z \,,
$$
and, if $r<\delta/2$, we have, when $\inf_\omega d_\Omega \geq \delta$, that
$$
(d_\Omega(x)-2r)^\alpha (d_\Omega(x)+ 2(\sgn\beta) r)^{-\beta}
\leq (1- 2r\delta^{-1})^\alpha (1+ 2(\sgn\beta)r\delta^{-1})^{-\beta} d_\Omega(x)^{\alpha-\beta} \,.
$$
Thus,
\begin{align*}
	& (d_\Omega(z)-r)^\alpha (d_\Omega(z)+ (\sgn\beta) r)^{-\beta} \left( \int_{\omega\cap\supp\chi_z} b^\frac d2 \frac{dx}{\sqrt{\det a}} \right)^{-\frac 2d} \\
	&\quad  \geq (1- 2r\delta^{-1})^\alpha (1+ 2(\sgn\beta)r\delta^{-1})^{-\beta} \left( \int_{\omega\cap\supp\chi_z} d_\Omega^{-(\alpha -\beta)d/2} b^\frac d2 \frac{dx}{\sqrt{\det a}} \right)^{-\frac 2d}\,.
\end{align*}
To summarize, we have shown that, under the assumptions $d_\Omega(z)>r$, $r\leq\rho$ and $r<\delta/2$, we have
\begin{align*}
	& \int_\omega d_\Omega^\alpha \nabla(\chi_z u) \cdot a\nabla(\chi_z u)\,dx \\
	& \quad \geq (1-\epsilon)^3 (1- 2r\delta^{-1})^\alpha (1+ 2(\sgn\beta)r\delta^{-1})^{-\beta} C^{\rm FK} \mu(\omega)^{-\frac 2d} \int_\omega d_\Omega^\beta b |\chi_z u|^2\,dx \,.
\end{align*}

Continuing to assume $r<\delta/2$, we see that the same bound holds trivially when $d_\Omega(z) \leq r$, because in this case we have $\chi_z u=0$. Indeed, $\supp\chi_z \subset\{ d_\Omega \leq r\} + \overline{B_r} \subset \{ d_\Omega \leq 2r \}$ and $\inf_\omega d_\Omega \geq\delta$.

Therefore we can integrate the bound with respect to $z\in\R^d$ and we obtain
$$
\begin{array}{l}
	\int_{\R^d} \int_\omega d_\Omega^\alpha \nabla(\chi_z u) \cdot a\nabla(\chi_z u)\,dx \,dz \\
	\qquad \geq (1-\epsilon)^3 (1- 2r\delta^{-1})^\alpha (1+ 2(\sgn\beta)r\delta^{-1})^{-\beta} C^{\rm FK} \mu(\omega)^{-\frac 2d} \int_\omega d_\Omega^\beta b | u|^2\,dx \,.
\end{array}
$$
Inserting this into the IMS formula gives
\begin{align*}
	\int_\omega d_\Omega^\alpha \nabla u\cdot a \nabla u\,dx & \geq \left( (1-\epsilon)^3 (1- 2r\delta^{-1})^\alpha (1+ 2(\sgn\beta)r\delta^{-1})^{-\beta} C^{\rm FK} - \eta^\frac2d C_{r,\delta} \right) \\
	& \quad \times \mu(\omega)^{-\frac 2d} \int_\omega d_\Omega^\beta\, b \,  | u|^2\,dx \,.
\end{align*}

It is now easy to finish the proof. We recall that the parameters $\theta>0$ and $\delta>0$ are given. We first choose $r>0$ and $\epsilon>0$ so small that
$$
(1-\epsilon)^3 (1- 2r\delta^{-1})^\alpha (1+ 2(\sgn\beta)r\delta^{-1})^{-\beta} \geq 1- \theta/2
$$
and then we choose $\eta>0$ so small that
$$
\eta^\frac2d C_{r,\delta} (C^{\rm FK})^{-1} \leq \theta/2 \,.
$$
This completes the proof of Lemma \ref{almostfkdeg}.
\qed


\section{Uniform Faber--Krahn inequalities}\label{s5}

In this section we will prove the uniform Faber--Krahn inequalities stated  in Lemmas \ref{fkunif} and \ref{fkunifdeg}. In contrast to the inequalities in the previous section, they are valid for arbitrary open subsets, but with a possibly much worse constant.\\

\subsection{The uniform Faber--Krahn inequality in the Baouendi--Grushin case}~\\
In this subsection we prove Lemma \ref{fkunif}. We rely on Sobolev-type inequalities for the Baouendi--Grushin operator. These inequalities involve specific weights and we have not been able to find them in the literature, so we provide complete proofs. We recall the definition of the Baouendi--Grushin Sobolev spaces $S^{1,(\alpha)}(\R^{n+m})$ from Subsection~\ref{sec:strategybg}. We need two different arguments, one that works when $n=1$ and one that works when $n+m>2$. The resulting Sobolev inequalities are summarized as follows.

\begin{proposition}\label{sobint}
	Let $m\geq 1$. There is a constant $S>0$ such that for all $\alpha\geq 0$ and all $u\in S^{1,(\alpha)}(\R^{1+m})$,
	$$
	\left( \int_{\R^{1+m}} |\nabla^{(\alpha)} u|^2\,dx \right)^{1/2} \left( \int_{\R^{1+m}} |u|^2\,dx \right)^{1/2} \geq S \left( \int_{\R^{1 +m}} |u|^\frac{2(1+m)}{m} |x_1|^\alpha \,dx \right)^\frac{m}{1+m}.
	$$
\end{proposition}

\begin{proposition}\label{sob}
	Let $n,m\geq 1$ with $n+m>2$ and $\alpha\geq 0$. Then there is a constant $S>0$ such that for all $u\in S^{1,(\alpha)}(\R^{n+m})$,
	$$
	\int_{\R^{n+m}} |\nabla^{(\alpha)}u|^2\,dx \geq S \left( \int_{\R^{n+m}} |u|^\frac{2(n+m)}{n+m-2} |x_1|^\frac{2m\alpha}{n+m-2} \,dx \right)^\frac{n+m-2}{n+m}.
	$$
\end{proposition}

Before giving the proof of these two propositions, let us show how they imply Lemma~\ref{fkunif}.

\begin{proof}[Proof of Lemma \ref{fkunif}]
	Let $\omega\subset\R^{n +m}$ be an open set with $\mu(\omega)<\infty$ and let $u\in S^{1,(\alpha)}_0(\omega)$. Note that its extension by zero belongs to $S^{1,(\alpha)}(\R^{n+m})$.
	
	We distinguish the cases $n=1$ and $n+m>2$. (These cases are not exclusive, but they cover all possibilities.) If $n=1$, we have, by H\"older's inequality,
	$$
	\int_\omega |u|^2\,dx \leq \left( \int_\omega \frac{dx}{|x_1|^{\alpha m}} \right)^\frac 1{1+m} \left( \int_\omega |u|^\frac{2(1+m)}{m} |x_1|^\alpha\,dx \right)^\frac{m}{1+m}.
	$$
	Bounding the second term on the right side using Proposition \ref{sobint}, we obtain
	$$
	\int_\omega |u|^2\,dx \leq S^{-1} \left( \int_\omega \frac{dx}{|x_1|^{\alpha m}} \right)^\frac 1{m+1} \left( \int_{\R^{1+n}} |\nabla^{(\alpha)} u|^2\,dx \right)^{1/2} \left( \int_{\R^{1+n}} |u|^2\,dx \right)^{1/2}.
	$$
	Equivalently,
	$$
	\int_\omega |u|^2\,dx \leq S^{-2} \left( \int_\omega \frac{dx}{|x_1|^{\alpha m}} \right)^\frac 2{m+1} \int_{\R^{1+n}} |\nabla^{(\alpha)} u|^2\,dx \,,
	$$
	which proves the claimed bound.
	
	If $n+m>2$, we have, by H\"older's inequality,
	$$
	\int_\omega |u|^2\,dx \leq \left( \int_\omega \frac{dx}{|x_1|^{\alpha m}} \right)^\frac 2{n+m} \left( \int_\omega |u|^\frac{2(n+m)}{n+m-2} |x_1|^\frac{2m\alpha}{n+m-2}\,dx \right)^\frac{n+m-2}{n+m}
	$$
	and we can bound the second term on the right side using Proposition \ref{sob}.
\end{proof}

\begin{proof}[Proof of Proposition \ref{sobint}]
	Let $I(\R^d)$ denote the constant in the isoperimetric inequality on $\R^d$, that is,
	$$
	I(\R^d) = d^{(d-1)/d} \, |\Sph^{d-1}|^{1/d}.
	$$
	
	\emph{Step 1.} We shall prove the \emph{sharp} inequality
	$$
	\int_{\R^{1+m}} |\nabla^{(\alpha)} u|\,dx \geq I(\R^{1+m}) \left( \int_{\R^{1+m}} |u|^\frac{1+m}{m} |x_1|^\alpha \,dx \right)^\frac{m}{1+m}
	$$
	for, say, $u\in C^1_c(\R^{1+m})$. Indeed, given $u$, define $v\in C^1_c(\R^{1+m})$ by
	$$
	u(x_1,x_2) = v((\alpha+1)^{-1} |x_1|^{\alpha} x_1,x_2) \,.
	$$
	Then a short computation shows that
	\begin{equation}
		\label{eq:covgrad}
			|(\nabla^{(\alpha)}u)(x)| = (\alpha+1)^{-1} |x_1|^\alpha |(\nabla v)((\alpha+1)^{-1} |x_1|^{\alpha} x_1,x_2)| \,.
	\end{equation}
	Therefore, by a change of variables,
	$$
	\int_{\R^{1+m}} |\nabla^{(\alpha)} u|\,dx = \int_{\R^{1+m}} |\nabla v| \,dy \,.
	$$
	The same change of variables also shows that
	$$
	\int_{\R^{1+m}} |u|^\frac{n+m}{n+m-1} |x_1|^\alpha \,dx = \int_{\R^{1+m}} |v|^\frac{1+m}{m} \,dy \,.
	$$
	Therefore, the inequality claimed at the beginning of this step follows from the well-known `isoperimetric inequality'
	$$
	\int_{\R^{1+m}} |\nabla v|\,dy \geq I(\R^{1+m}) \left( \int_{\R^{1+m}} |v|^\frac{1+m}{m} \,dy \right)^\frac{m}{1+m}.
	$$

	\medspace
	
	\emph{Step 2.} In the inequality in Step 1, we replace $u$ by $u^2$ and note that
	$$
	|\nabla^{(\alpha)} u^2|  \leq 2 \, |u| \, |\nabla^{(\alpha)} u| \,.
	$$
	This, together with the Schwarz inequality, implies the bound in the proposition with $S=2^{-1} I(\R^{1+m})$.
\end{proof}

The following proof draws some inspiration from work by M. Rumin \cite{Ru}, who in turn says that he is motivated by work of Chemin and Xu \cite{CX}.

\begin{proof}[Proof of Proposition \ref{sob}]
	We write $\mathcal L = -\Delta_1 - |x_1|^{2\alpha} \Delta_2$ for the operator in $L^2(\R^{n+m})$ and note that
	\begin{equation}
		\label{eq:ruminrepresentation}
			\int_{\R^{n+m}} |\nabla^{(\alpha)} u|^2\,dx = \|\mathcal L^{1/2} u\|^2 = \int_{\R^{n+m}} \int_0^\infty |u^{>E}(x)|^2 \,dE\,dx \,,
	\end{equation}
	where
	$$
	u^{>E} := \1(\mathcal L>E) u \,.
	$$
	Indeed, by the spectral theorem, we have
	$$
	\|\mathcal L^{1/2} u\|^2 = \int_0^\infty \| \1(\mathcal L>E) u \|^2 \,dE = \int_0^\infty \| u^{>E} \|^2 \,dE \,,
	$$
	so that the second identity in \eqref{eq:ruminrepresentation} follows by Fubini's theorem.
	
	Defining $u^{\leq E}$ in the obvious way, we have
	\begin{equation}
		\label{eq:triangle}
			|u(x)|\leq |u^{\leq E}(x)| + |u^{>E}(x)| \,.
	\end{equation}
	We write
	\begin{align*}
		|u^{\leq E}(x)|^2 & = \left| \int_{\R^{n+m}} (\mathcal L^{-1/2}\1(\mathcal L\leq E))(x,x') (\mathcal L^{1/2}u)(x')\,dx' \right|^2 \\
		& \leq \int_{\R^{n+m}} |(\mathcal L^{-1/2}\1(\mathcal L\leq E))(x,x')|^2 \,dx' \, \|\mathcal L^{1/2} u\|^2 \\
		& = (\mathcal L^{-1}\1(\mathcal L\leq E))(x,x) \, \|\mathcal L^{1/2} u\|^2 \,.
	\end{align*}
	
	We shall show below that there exists $C>0$ such that
	\begin{equation}
		\label{eq:kernelbound}
		0 \leq (\mathcal L^{-1}\1(\mathcal L\leq E))(x,x) \leq C E^\frac{n+m-2}{2}\, |x_1|^{-m\alpha} \,.
	\end{equation}
	Accepting this for the moment, we obtain from \eqref{eq:triangle} that
	$$
	|u^{>E}(x)| \geq \left( |u(x)| - C^{1/2} E^\frac{n+m-2}{4} |x_1|^{-\frac{m\alpha}2} \|\mathcal L^{1/2} u\| \right)_+
	$$
	and, consequently, by a change of variables, there exists $C' >0$ such that
	$$
	\int_0^\infty |u^{>E}(x)|\,dE = C'\, |u(x)|^\frac{2(n+m)}{n+m-2} \, |x_1|^\frac{2m\alpha}{n+m-2}\, \|\mathcal L^{1/2} u\|^{-\frac 4{n+m-2}} \,.
	$$
	Here $C'$ is given explicitly in terms of $C$, $m$, $n$ and $\alpha$.
	Inserting this into \eqref{eq:ruminrepresentation} gives
	$$
	\|\mathcal L^{1/2} u\|^2 \geq C' \,\|\mathcal L^{1/2} u\|^{-\frac 4{n+m-2}} \int_{\R^{n+m}}\,  |u(x)|^\frac{2(n+m)}{n+m-2} \, |x_1|^\frac{2m\alpha}{n+m-2} \,dx \,,
	$$
	which is the claimed Sobolev-type inequality.
	
	It remains to prove \eqref{eq:kernelbound}. The left inequality is clear, since the operator $\mathcal L^{-1}\1(\mathcal L\leq E)$ is nonnegative. The main tool for the proof of the right inequality is the heat kernel bound
	\begin{equation}
		\label{eq:heatkernelbound0}
			0 \leq e^{-t\mathcal L}(x,x) \leq K\, t^{-\frac{n+m}{2}} (\sqrt t + |x_1|)^{-m\alpha}
	\end{equation}
	with a constant $K$ that depends on $n$, $m$ and $\alpha$. This bound is from \cite{RoSi}. Indeed, Theorem 6.4 in that reference shows that $e^{-t\mathcal L}(x,x)$ is bounded by a constant, depending on $n$, $m$, and $\alpha$, multiplied by  $|B(x;t^{1/2})|^{-1}$, where the adapted balls are defined at the beginning of Section 5. Bounds on this volume are given in Proposition 5.1. We note that we need these results for $\delta_1 = \delta_1' = 0$ and $\delta_2 = \delta_2' = \alpha$. Consequently, we have $D=D' = n+(1+\alpha)m$; see Proposition 3.1. The bound in Proposition 5.1 reads
	$$
	|B(x;r)| \sim
	\begin{cases}
		r^{n+(1+\alpha)m} & \text{if}\ r\geq |x_1| \,,\\
		r^{n+m} |x_1|^{m\alpha} & \text{if}\ r\leq |x_1| \,.
	\end{cases}
	$$
	Thus, we have $|B(x;r)| \sim r^{n+m} (r+ |x_1|)^{m\alpha}$, which, for $r=\sqrt t$, gives the claimed bound \eqref{eq:heatkernelbound0}
	
	In fact, in what follows, we shall only use the weaker bound that there exists $K>0$ such that
	\begin{equation}
		\label{eq:heatkernelbound}
			0 \leq e^{-t\mathcal L}(x,x) \leq K t^{-\frac{n+m}{2}} |x_1|^{-m\alpha} \,.
	\end{equation}
	(If, instead, we would use the upper bound $Kt^{-\frac{n+(1+\alpha)m}2}$ in the following proof, we would arrive at the standard unweighted Sobolev inequality for the Grushin operator.)
	
	To prove \eqref{eq:kernelbound} we write
	$$
	\mathcal L^{-1} \1(\mathcal L\leq E) = \int_0^\infty e^{-u\mathcal L} \1(\mathcal L\leq E)\,du\,,
	$$
	and we split the integral at $u=1/E$. For $u\geq 1/E$ we use the inequality
	$$
	e^{-u\mathcal L} \1(\mathcal L\leq E) \leq e^{-u\mathcal L}\,, 
	$$
	which implies
	$$
	\left( e^{-u\mathcal L} \1(\mathcal L\leq E) \right)(x,x) \leq e^{-u\mathcal L}(x,x) \leq K u^{-\frac{n+m}{2}} |x_1|^{-m\alpha} \,.
	$$
	For $u< 1/E$, we use the inequality
	$$
	e^{-u\mathcal L}\1(\mathcal L\leq E) \leq \1(\mathcal L\leq E) \leq e \, e^{-(1/E)\mathcal L} \,,
	$$
	which implies
	$$
	\left( e^{-u\mathcal L} \1(\mathcal L\leq E) \right)(x,x) \leq e K E^\frac{n+m}2 |x_1|^{-m\alpha} \,.
	$$
	Thus,
	$$
	\left( \mathcal L^{-1} \1(\mathcal L\leq E) \right)(x,x) \leq \int_0^{1/E} e K E^\frac{n+m}2 |x_1|^{-m\alpha} \,du + \int_{1/E}^\infty K u^{-\frac{n+m}{2}}  |x_1|^{-m\alpha} \,du \,.
	$$
	Computing the integrals, we arrive at \eqref{eq:kernelbound}.
\end{proof}


\bigskip

\subsection{The uniform Faber--Krahn-type inequality in the case of boundary degeneration}~\\
In this subsection we prove Lemma \ref{fkunifdeg}, using the following weighted Sobolev-type inequalities. We recall the definition of the weighted Sobolev spaces $H^{1,(\alpha,\beta)}_0(\Omega)$ from Subsection \ref{sec:strategydeg}.

\begin{proposition}\label{sobintdegdom}
	Let $\Omega\subset\R^2$ be a bounded open set with $C^{0,1}$-boundary. Let $\alpha\in[0,1)$, $\beta\in\R$ and $q\in[2,\infty)$. Then there is a constant $S>0$ such that, for all $u\in H^{1,(\alpha,\beta)}_0(\Omega)$,
	$$
	\left( \int_\Omega d_\Omega^\alpha\, |\nabla u|^2\,dx \right)^\frac{q-2}q \left( \int_\Omega d_\Omega^\beta\, |u|^2\,dx \right)^\frac 2q \geq S \left( \int_\Omega d_\Omega^{\beta+\alpha(\frac q2-1)} |u|^q  \,dx \right)^\frac2q .
	$$
\end{proposition}

\begin{proposition}\label{sobdegdom}
	Let $d\geq 3$ and $\alpha\in [0,1)$. Let $\Omega\subset\R^d$ be a bounded open set with $C^{0,1}$-boundary. Then there is a constant $S>0$ such that, for all $u\in H^{1,(\alpha,\beta)}_0(\Omega)$,
	$$
	\int_\Omega d_\Omega^\alpha \, |\nabla u|^2\,dx \geq S \left( \int_\Omega d_\Omega^\frac{\alpha d}{d-2} \, |u|^\frac{2d}{d-2} \,dx \right)^\frac{d-2}d.
	$$
\end{proposition}

Before giving the proof of these two propositions, let us show that they imply the Faber--Krahn-type inequality.

\begin{proof}[Proof of Lemma \ref{fkunifdeg}]
	Let $d\geq 3$ and let $\alpha\geq 0$ and $\beta>-2$ with $\alpha-\beta<2/d$. If $\omega\subset\Omega$ is open and $u\in H^{1,(\alpha,\beta)}_0(\omega)$, then, by H\"older's inequality,
	$$
	\int_\omega d_\Omega^\beta \, |u|^2\,dx \leq \left( \int_\omega d_\Omega^{- \frac{d(\alpha-\beta)}2}\,dx \right)^\frac2d \left( \int_\omega d_\Omega^\frac{\alpha d}{d-2} |u|^\frac{2d}{d-2} \,dx \right)^\frac{d-2}{d}.
	$$
	According to Proposition \ref{sobdegdom} we have
	$$
	\left( \int_\omega d_\Omega^\frac{\alpha d}{d-2}\, |u|^\frac{2d}{d-2} \,dx \right)^\frac{d-2}{d}
	\leq S^{-1} \int_\omega d_\Omega^\alpha \, |\nabla u|^2\,dx \,.
	$$
	Combining these two bounds gives the claimed inequality for $d\geq 3$.
	
	If $d=2$, using again H\"older's inequality,  we bound for some arbitrary $q\in(2,\infty)$
	$$
	\int_\omega d_\Omega^\beta\,  |u|^2\,dx \leq \left( \int_\omega d_\Omega^{-\alpha+\beta}\,dx \right)^{1-\frac 2q} \left( \int_\omega d_\Omega^{\beta+\alpha(\frac q2-1)} |u|^q \,dx \right)^\frac 2q.
	$$
	Bounding the second term on the right side using Proposition \ref{sobintdegdom}, we arrive at the claimed inequality for $d=2$.
\end{proof}

For the proof of Propositions \ref{sobintdegdom} and \ref{sobdegdom}, we need the following weighted Sobolev inequalities from \cite[Theorem 21.3]{OK}. For $1\leq p<d$ and $\gamma\in(-\infty,p-1)$ we let $W^{1,p,(\gamma)}_0(\Omega)$ denote the completion of $C^1_c(\Omega)$ with respect to the norm $(\int_\Omega d_\Omega^\gamma\,  |\nabla u|^p\,dx)^{1/p}$.

\begin{theorem}\label{opickufner}
	Let $1\leq p<d$ and let $\gamma\in(-\infty,p-1)$. Let $\Omega\subset\R^d$ be a bounded open set with $C^{0,1}$-boundary. Then there is a constant $S>0$ such that, for all $u\in W^{1,p,(\gamma)}_0(\Omega)$,
	$$
	\int_\Omega d_\Omega^\gamma\, |\nabla u|^p\,dx \geq S \left( \int_\Omega d_\Omega^{\frac{d\gamma}{d-p}} |u|^\frac{dp}{d-p}\,dx \right)^\frac{d-p}{d}.
	$$ 
\end{theorem}

\begin{proof}[Proof of Proposition \ref{sobdegdom}]
	This follows immediately from Theorem \ref{opickufner} with $p=2$.
\end{proof}

\begin{proof}[Proof of Proposition \ref{sobintdegdom}]
	We first consider the case $q>2 (\beta-\alpha +2)/(1-\alpha)$ and $q\geq 4$. We apply Theorem \ref{opickufner} with
	$$
	d= 2 \,,
	\qquad
	p = \frac{2(q-2)}{q}
	\qquad\text{and}\qquad
	\gamma = \frac{2}{q} \left( \beta + \alpha\left(\frac q2 -1\right) \right).
	$$
	The assumption $1\leq p <2$ is satisfied in view of $q\geq 4$ and the assumption $\gamma<p-1$ is satisfied in view of $q>2 (\beta-\alpha +2)/(1-\alpha)$. Thus, we obtain
	$$
	\int_\Omega |\nabla v|^{2-\frac 4q} \, d_\Omega^{\frac{2}{q} \left( \beta + \alpha\left(\frac q2 -1\right) \right)}\,dx \geq S \left( \int_\Omega |v|^{q-2}\, d_\Omega^{\beta + \alpha(\frac q2-1)}\,dx \right)^\frac{2}{q}.
	$$
	We choose $v= |u|^{q/(q-2)}$ and obtain the existence of $\tilde S >0$ such that
	$$
	\int_\Omega |u|^{\frac 4q}\, |\nabla u|^{2-\frac 4q} \, d_\Omega^{\frac{2}{q} \left( \beta + \alpha\left(\frac q2 -1\right) \right)}\,dx \geq \tilde S  \left( \int_\Omega |u|^q \, d_\Omega^{\beta + \alpha(\frac q2-1)}\,dx \right)^\frac{2}{q}. 
	$$	
	By H\"older's inequality, writing
	$$
	|u|^{\frac 4q}\, |\nabla u|^{2-\frac 4q} \,d_\Omega^{\frac{2}{q} \left( \beta + \alpha\left(\frac q2 -1\right) \right)} = \left( |u|^2 d_\Omega^\beta \right)^{\frac{2}q} \left( |\nabla u|^2 d_\Omega^\alpha \right)^\frac{q-2}{q},
	$$
	we find
	$$
		\int_\Omega |u|^{\frac 4q}\, |\nabla u|^{2-\frac 4q}\, d_\Omega^{\frac{2}{q} \left( \beta + \alpha\left(\frac q2 -1\right) \right)}\,dx 
		\leq \left( \int_\Omega |u|^2 d_\Omega^\beta \,dx \right)^{\frac{2}q} \left( \int_\Omega |\nabla u|^2 d_\Omega^\alpha \,dx \right)^\frac{q-2}{q}.
	$$
	To summarize, we have shown the existence of $\check S >0$ such that
	$$
	\left( \int_\Omega |u|^2 d_\Omega^\beta \,dx \right)^{\frac{2}q} \left( \int_\Omega |\nabla u|^2 d_\Omega^\alpha \,dx \right)^\frac{q-2}{q}
	\geq \check S \left( \int_\Omega |u|^q \, d_\Omega^{\beta + \alpha(\frac q2-1)}\,dx \right)^{ \frac{2}{q}}. 
	$$
	This is the claimed inequality under the assumptions $q>2(\beta-\alpha+2)/(1-\alpha)$ and $q\geq 4$.
	
	For general $q\in[2,\infty)$, we fix a $q_0>q$ such that $q_0>2(\beta-\alpha+2)/(1-\alpha)$ and $q_0\geq 4$ and we bound
	$$
	\int_\Omega |u|^q d_\Omega^{\beta + \alpha(\frac q2-1)}\,dx
	\leq \left( \int_\Omega |u|^{q_0} \, d_\Omega^{\beta + \alpha(\frac {q_0}2-1)}\,dx \right)^{\frac{q-2}{q_0-2}} \left( \int_\Omega |u|^2 d_\Omega^{\beta}\,dx \right)^{\frac{q_0-q}{q_0-2}}.
	$$
	Applying the bound for $q_0$ that we just proved, we obtain the claimed inequality in the general case. This completes the proof of Proposition \ref{sobintdegdom}.
\end{proof}


\section{Extensions and additional remarks}\label{sec:final}

So far, in this paper we have shown a Pleijel bound for two different classes of degenerate elliptic operators, namely the Baouendi--Grushin operators and operators with ellipticity degenerating on the boundary. We believe that the techniques that we have introduced could be applied to treat several other interesting cases of degenerate elliptic operators, and in the first part of this section we mention some possible extensions. In the second part we review the Baouendi--Grushin case in light of these extensions and we briefly discuss a familly of operators from \cite{CDHT}.


\subsubsection*{Extension 1}
Let $a$ and $b$ be as in Subsection \ref{s4} and assume that $\beta >-1$ and $\alpha\geq 0$. The quadratic form \eqref{eq:quadform} considered on $C^1(\overline\Omega)$ (rather than on $C^1_c(\Omega)$) generates an operator $H$ in the Hilbert space with norm \eqref{eq:norm}, which corresponds to natural (Neumann--type) boundary conditions. (As an aside we mention the fact \cite{Vi} that for $\alpha\geq 1$ the operators defined via the quadratic form on $C^1(\overline\Omega)$ and on $C^1_c(\Omega)$ coincide.)

We believe that, when $\alpha-\beta<2/d$, the analogue of Theorem \ref{maindeg} remains valid. 

The Weyl asymptotics in this case are known \cite{BiSo0,BiSo,Ta}, so what remains to be done is to prove the Faber--Krahn-type inequalities. We expect the Sobolev inequalities in \cite[Subsection 2.1.7]{Ma} to be useful for this task.

\subsubsection*{Extension 2}
One could consider a second order operator with ellipticity degenerating on a closed $C^1$-submanifold $\Sigma$ of $\overline\Omega$ of co-dimension 1 and the degeneration is comparable with $\dist(x,\Sigma)^\alpha$. Again, at the same time a degeneration of the inner product of the form $\dist(x,\Sigma)^\beta$ is allowed. Both the cases of Dirichlet and Neumann boundary conditions can be considered and in the latter case one probably wants to assume that $\Sigma$ is contained in $\Omega$.

We believe that when the parameters $\alpha$ and $\beta$ satisfy the assumptions of Extension 1, then Theorem \ref{maindeg} remains valid.

\subsubsection*{Extension 3}
We believe that similar results are valid when the closed submanifold $\Sigma$ of $\overline\Omega$ has arbitrary codimension. Now the condition $\alpha-\beta<2/d$ needs to be adjusted correspondingly.

\subsubsection*{Another look at the Baouendi-Grushin case}
We consider general $n,m\geq 1$, although later we will restrict ourselves to $n=1$. We consider an open set $\Omega\subset\R^{n+m}$ and recall that the quadratic form under consideration is
$$
\int_\Omega (|\nabla_1 u|^2 + |x_1|^{2a}|\nabla_2u|^2)\,dx
$$
with $0\leq a< n/m$. (We write $a$ in order to avoid confusion with the parameter $\alpha$ that appears later.) Assume $u$ and $v$ are related by
$$
u(x_1,x_2) = v( (a+1)^{-1} |x_1|^a x_1,x_2) \,.
$$
Then
\begin{align*}
	|\nabla_1 u(x_1,x_2)|^2 = |x_1|^{2a} & \left( (\partial_{s_1} v)( (a+1)^{-1} |x_1|^a x_1,x_2)^2 \right. \\
	& \quad \left. + \ \frac{(a+1)^{-2}}{|(a+1)^{-1}|x_1|^a x_1|^2} |(\nabla_{\theta_1}v)((a+1)^{-1} |x_1|^a x_1,x_2)|^2 \right)
\end{align*}
where we consider $v$ as a function of variables $(y_1,y_2)\in\R^{n+m}$ and write $|y_1|=s_1$ and $y_1/|y_1|=\theta_1$. Thus, by a change of variables,
\begin{align*}
	& \int_\Omega (|\nabla_1 u|^2 + |x_1|^{2a}|\nabla_2u|^2)\,dx \\
	& = (a+1)^{\frac{n-1+a}{a+1}} \int_{\tilde\Omega} \left( |\partial_{s_1} v|^2 + \frac{(a+1)^{-2}}{|y_1|^2}|\nabla_{\theta_1} v|^2 + |\nabla_2 v|^2 \right) |y_1|^{\frac{a(2-n)}{a+1}}\,dy \,,
\end{align*}
where
$$
\tilde\Omega := \{ ((a+1)^{-1}|x_1|^a x_1,x_2) \in\R^{n+m}:\ (x_1,x_2)\in\Omega \} \,.
$$
Similarly,
$$
\int_\Omega |u|^2\,dx = (a+1)^{\frac{n-1-a}{a+1}} \int_{\tilde\Omega} |v|^2 |y_1|^{-\frac{an}{a+1}}\,dy \,.
$$

In particular, when $n=1$, this is an operator essentially of the same form as in Theorem \ref{maindeg}, except that the distance $d_\Omega$ is replaced by the distance $\dist(x,\Sigma)$ to the hypersurface
$$
\Sigma = \{ x\in\Omega :\ x_1 = 0 \} \,.
$$
Moreover, we have
$$
\alpha = \frac{a}{a+1}
\qquad\text{and}\qquad
\beta = - \frac{a}{a+1}
$$
and, recalling that $0\leq a<1/m$, we see that the assumptions $0\leq\alpha<1$, $\beta >-2$ and $\alpha-\beta<2/d$ (with $d:=1+m$) in Theorem \ref{maindeg} are satisfied. Therefore, if Theorem \ref{maindeg} remains valid in the case where the degeneration occurs on any hypersurface (Extension 2), then Theorem \ref{main} is a consequence of this extension of Theorem \ref{maindeg} (except, possibly, for the assumptions on $\Omega$).

If Theorem \ref{maindeg} remained valid for degenerations along submanifolds of arbitrary co-dimension (Extension 3) and with \emph{negative} degeneration exponents $\alpha$, then Theorem \ref{main} would be a special case also with $n\geq 2$.


\subsubsection*{The operator from \cite{CDHT}}

The setting is a smooth compact manifold $X$ of dimension $d\geq 2$ with boundary. We assume that near the boundary, $X$ is diffeomorphic to $[0,1)\times M$, where $M$ is a smooth compact manifold of dimension $m:=d-1$. We assume that there is a transverse coordinate $\rho$, ranging over $[0,1)$, such that $\partial X = \{ \rho =0 \}$. We assume that there is a Riemannian metric on the interior of $X$ such that on $[0,1)\times M$ we have
$$
g = \rho^{-a} (d\rho^2 + g_0(\rho))
$$
with an exponent $0<a<2/d$ and with a family of Riemannian metrics $g_0(\rho)$ on $M$ that depends smoothly on $\rho\in[0,1)$. Consequently, the Riemannian volume form on $[0,1)\times M$ can be written as $dv_g = \rho^{-ad/2} d\rho\,dv_{g_0(\rho)}$, where for each $\rho\in[0,1)$, $dv_{g_0(\rho)}$ is a Riemannian volume form on $M$. For functions $u$ supported on $[0,1)\times M$ the square of the Hilbert space norm is
$$
\int_0^1 \int_M |u|^2 \rho^{-ad/2}\,dv_{g_0(\rho)} \,d\rho
$$
and the quadratic form of the Laplace--Beltrami operator is
$$
\int_0^1 \int_M \left( |\partial_\rho u|^2 + |\overline d u|_{g_0(\rho)}^2 \right) \rho^{a(2-d)/2}\,dv_{g_0(\rho)} \,d\rho
$$
where $\overline d$ is the differential on $M$. The quadratic form is defined for functions $u\in C^1_c(X\setminus\partial X)$.

Thus, we have essentially a problem as in Theorem \ref{maindeg} with
$$
\alpha = \frac{a(2-d)}{2} 
\qquad
\beta = -\frac{ad}{2} \,.
$$
The assumption $\alpha - \beta <2/d$ is satisfied in view of $a<2/d$, and this also implies $\beta>-1>-2$. However, just as in the Baouendi--Grushin case, when $d\geq 3$, the assumption $\alpha\geq 0$ is \emph{not} satisfied.

Let us argue that the proof of Theorem \ref{maindeg} goes through in the present setting and gives a Pleijel bound with constant $\gamma(\R^d)$. The assumption $\alpha\geq 0$ comes mainly from the Weyl asymptotics in Lemma \ref{weyldeg}, but the use of this proposition can be replaced by the main result in \cite{CDHT}. We believe that the assumption $\alpha\geq 0$ is not needed in the proofs of Propositions \ref{almostfkdeg} and \ref{fkunifdeg}. In the proof of Proposition \ref{almostfkdeg}, only a minor change arises at the end of the proof when $(d_\Omega(z)-r)^\alpha$ is bounded from below. (The lower bound $(1-2r\delta)^\alpha$ needs to be replaced by $(1-2(\sgn\alpha) r\delta)^{\alpha}$.) In the proof of Proposition \ref{fkunifdeg} the main ingredient is Theorem \ref{opickufner}, and this result is valid for weights $A\in(-\infty,p-1)$. Because of this, Proposition \ref{sobdegdom} extends immediately to $\alpha\in(-\infty,0)$.


\section{On Courant's nodal theorem for degenerate elliptic operators}\label{sec:courant}

\subsection{The case of the Baouendi--Grushin operator}~\\
Our goal in this subsection is to prove the Courant theorem for the Baouendi--Grushin operator, as stated in Theorem \ref{maincourant}. We recall that $\Omega$ is an open, connected subset of $\R^{n+m}$ and that we are assuming $n>1$. A partial result for $n=1$ is obtained in the next subsection. The difference between $n=1$ and $n>1$  comes from the fact that $\Omega\setminus\{x_1=0\}$ is connected in the latter case, but not in the first case.

\begin{proof}[Proof of Theorem \ref{maincourant}] 
	We shall write $\Sigma :=\{0\}\times\R^m = \{x_1=0\}$ for the singular set. Let $\psi$ be an eigenfunction corresponding to the eigenvalue $\lambda_k$ for some $k\geq 1$. We need to show that $\psi$ has at most $k$ nodal domains.
	
	We can use the usual contradiction argument assuming $\psi$ has more than $k$ nodal domains. We fix a nodal domain $\omega$ and construct a function $\widetilde\psi = \sum_j c_j \1_{\omega_j} \psi$ where $\omega_j$ are the nodal domains of $\psi$ \emph{different from $\omega$} and where the $c_j$ are numbers such that $\|\widetilde \psi\|_{L^2}=1$ and such that $\widetilde\psi$ is orthogonal to $k-1$ linearly independent eigenfunction corresponding to eigenvalues that do not exceed $\lambda_k$. As discussed for instance in \cite{FH}, the function $\widetilde\psi$ belongs to the form domain of the operator and satisfies the eigenvalue equation in the weak sense. It follows that $\widetilde\psi$ belongs to the operator domain and is an eigenfunction. By construction, $\widetilde\psi$ vanishes in $\omega$. Since $\Omega\setminus\Sigma$ is connected (because $n>1$), we can apply the unique continuation principle in the set $\Omega\setminus\Sigma$ and deduce that $\widetilde\psi$ vanishes in $\Omega\setminus\Sigma$, contradicting the assumption $\|\widetilde \psi\|_{L^2}=1$.
\end{proof}


\medskip

\subsection{Additional remarks}~\\
We end this section with further results on Courant nodal theorems for degenerate second-order elliptic operator in a connected open set $\Omega\subset\R^d$. We begin by noting that the standard proof (as presented for instance in the proof of Theorem \ref{maincourant}) works if one of the following assumptions holds:
\begin{enumerate}
\item[(a)] The coefficients are analytic in $\Omega$
\item[(b)] $\Omega \subset \mathbb R^2$
\item[(c)] The degeneracy is at the boundary (see the model considered in  \eqref{hypdeg})
\item[(d)] The degeneracy occurs on a set $\Sigma$ of measure $0$ and $\Omega\setminus \Sigma$ is connected (see the model considered in Section \ref{s1} when $n>1$ or the discussion in \cite{A}).
\end{enumerate}
Indeed, in case (a) one uses the unique continuation theorem by J.~M.~Bony \cite{Bo}, while in case (b) one uses those of Watanabe \cite{Wa}; see \cite{EL}. In cases (c) and (d) one needs to assume that the highest order coefficients are locally Lipschitz continuous in $\Omega$ and in $\Omega\setminus\Sigma$, respectively, so that one can use the unique continuation theorem of Aronszajn, Krzywicki and Szarski \cite{AKS}. Note that case (c) includes the case discussed in Subsection \ref{s4} and case (d) includes the case $n>1$ of the Baouendi--Grushin operator discussed in Subsection \ref{s1a}.

A case that is not covered by these results is where $\Sigma$ is a regular hypersurface in $\Omega$. In this case one might be tempted to apply a statement given by Bahouri \cite{Ba} as a side-remark, but no proof\footnote{This was confirmed by the author.} of this statement is available. In what follows we show that in this case it is possible to get a Courant nodal theorem in a weaker form, in the spirit of earlier work by Alessandrini \cite{Al}; see also \cite{SR}. We note this result is applicable in the case $n=1$ and $\alpha\in(0,\infty)\setminus\N$ of the Baouendi--Grushin operator, which is left open in Theorem \ref{maincourant}. We state the result for a general selfadjoint linear second order degenerate elliptic operator $L$ with $\Sigma$ denoting the set of degeneracy.

\begin{proposition}\label{weakcourant}
	Suppose that $\Omega$ is connected and that $\Sigma$ is a regular hypersurface such that $\Omega \setminus \Sigma= \Omega_1 \cup \Omega_2$ with $\Omega_j$ connected. If the operator $L$ satisfies the unique continuation property in $\Omega_1$ and $\Omega_2$, then any $k$-th eigenfunction ($k\geq 1$) has at most $(k+1)$ nodal domains.
\end{proposition}

We emphasize that the main thrust of this theorem is that we do \emph{not} assume the unique continuation property in $\Omega$, but only in $\Omega_1$ and $\Omega_2$. This weaker assumption leads to a slightly weaker bound on the number of nodal domains than in Courant's original theorem, namely $k+1$ rather than $k$. For the unique continuation property in $\Omega_j$ it suffices that the highest order coefficients are locally Lipschitz continuous in $\Omega_j$ \cite{AKS}.

 \begin{remark} 
 	The same proof shows that if $\Omega \setminus \Sigma=\bigcup_{j=1}^\ell   \Omega_j$ with $\Omega_j$ connected, then a $k$-th eigenfunction ($k\geq 1$)  has at most $(k+\ell -1)$ nodal domains.
\end{remark}

\begin{proof}
	Let $\psi$ be an eigenfunction corresponding to the eigenvalue $\lambda_k$, $k\geq 1$.
	
	\medskip
	
	\emph{Step 1.}	
	The easy case is when $\psi$ has a nodal domain $\omega$ that has a non-empty intersection with $\Sigma$. In this case we can show that $\psi$ has at most $k$ nodal domains. Indeed, in this case we can use the usual contradiction argument (see the proof of Theorem \ref{maincourant}) assuming $\psi$ has more than $k$ nodal domains and construct a normalized eigenfunction $\widetilde\psi$ that vanishes in $\omega$, and therefore in $\omega\cap\Omega_1$ and $\omega\cap\Omega_2$. The unique continuation theorem, applied separately in $\Omega_1$ and $\Omega_2$, implies that $\widetilde \psi$ vanishes in $\Omega_1$ and in $\Omega_2$, contradicting the assumption that $\widetilde\psi$ is normalized.
	
	\medskip
	
	\emph{Step 2.}
	It remains to deal with the case where every nodal domain of $\psi$ has empty intersection with $\Sigma$. In this case we follow the proof of \cite[Theorem 4.5]{Al}. We observe that $\psi|_{\Omega_j}$ is an eigenfunction (with eigenvalue $\lambda_k$) of the Dirichlet realization of $L$ in $\Omega_j$. If $K_j$ is the minimal labelling of $\lambda_k$ as an eigenvalue of this Dirichlet realization in $\Omega_j$, then by the minimax characterization of eigenvalues, we have
	\begin{equation}
		\label{eq:alessandrini}
			K_1 +  K_2 \leq k+1 \,.
	\end{equation}
	Indeed, let $\lambda_j^{(\Omega_j)}$, $j=1,\ldots, K_j-1$, be the eigenvalues below $\lambda_k$ of the Dirichlet realization in $\Omega_j$ and let $\phi_1,\ldots,\phi_{K_1-1}$ and $\chi_1,\ldots\chi_{K_2-1}$ be corresponding systems of orthonormal eigenfunctions in $\Omega_1$ and $\Omega_2$, respectively. Then for all real numbers $a_1,\ldots,a_{K_1-1}$ and $b_1,\ldots,b_{K_2-1}$, not all of them zero, the function
	$$
	u:= \sum_{j=1}^{K_1-1} a_j \phi_j + \sum_{j=1}^{K_2-1} b_j\chi_j
	$$
	(where the $\phi_j$ and $\chi_j$ are extended by zero to $\Omega$) belongs to the form domain of the operator $L$ in $\Omega$ and for the corresponding quadratic form, denoted by $h$, we have
	$$
	\frac{h[u]}{\| u \|^2} = \frac{\sum_{j=1} \lambda_j^{(\Omega_1)} a_j^2 \int_{\Omega_1} \phi_j^2\,dx + \sum_{j=1} \lambda_j^{(\Omega_2)} b_j^2 \int_{\Omega_2} \chi_j^2\,dx}{\sum_{j=1} a_j^2 \int_{\Omega_1} \phi_j^2\,dx + \sum_{j=1} b_j^2 \int_{\Omega_2} \chi_j^2\,dx} \leq \max\{ \lambda_{K_1-1}^{(\Omega_1)}, \lambda_{K_2-1}^{(\Omega_2)} \} < \lambda_k \,.
	$$
	Thus, for any choice of the $a_j$ and $b_j$, the function $u$ cannot be orthogonal to the space spanned by the eigenfunction $\psi_1,\ldots,\psi_{k'-1}$ where $k'$ is such that $\lambda_{k'-1}<\lambda_{k'} = \lambda_k$. It follows that
	$$
	(K_1 -1) + (K_2-1) < k' \,,
	$$
	which implies \eqref{eq:alessandrini} since $k'\leq k$.
	
	Since, by assumption, the unique continuation property holds in $\Omega_j$, Courant's theorem in $\Omega_j$ implies that the number $r_j$ of nodal domains of $\psi|_{\Omega_j}$ in $\Omega_j$ satisfies
	$$
	r_j \leq K_j \,.
	$$
	Since $r_1+r_2$ is the number of nodal domains of $\psi$ in $\Omega$, this inequality together with \eqref{eq:alessandrini} shows that $\psi$ has at most $k+1$ nodal domains.
\end{proof}

Under further assumptions on $L$, the bound from Proposition \ref{weakcourant} can be improved for small $k$.

\begin{proposition}
	In addition to the assumptions of Proposition \ref{weakcourant} assume that
	\begin{itemize}
		\item either $L = \sum_{j=1}^p X_j^* X_j$, where $X_j$ are smooth vector fields satisfying the H\"ormander condition,
		\item or $L$ is an operator of the form treated in \cite{FrLa2}.
	\end{itemize}
	Then any $k$-the eigenfunction with $k=1,2$ has exactly $k$ nodal domains.
\end{proposition}

We do not reproduce the precise assumptions of the operators treated in \cite{FrLa2}, but we emphasize that the smoothness of the vector fields there is relaxed and that this class includes the Baouendi--Grushin operators from Subsection \ref{s1a}.

\begin{proof}
	Let $\psi$ be an eigenfunction corresponding to the eigenvalue $\lambda_k$, $k =1,2$.
	
	According to Proposition \ref{weakcourant} and its proof, we know that $\psi$ has at most $k+1$ nodal domains and that the case of $k+1$ nodal domains can only occur if every nodal domain has empty intersection with $\Sigma$.
	
	If there was an $\ell\in\{1,2\}$ such that all nodal domains of $\psi$ are in $\Omega_\ell$, then we could follow the standard proof in $\Omega_\ell$ and would arrive at a contradiction.
	
	Thus, up to relabelling the $\omega_i$ and $\Omega_j$, we may assume that $\omega_1\subset\Omega_1$ and $\omega_2,\omega_{k+1}\subset\Omega_2$. Following the usual proof, we construct a normalized eigenfunction $\widetilde\psi$ for $\lambda_k$ that vanishes in $\omega_2$. Thus, for $k=1$ we see that $\widetilde\psi$ vanishes in $\Omega_2$, and we arrive at the same conclusion also for $k=2$ by using the unique continuation property in $\Omega_2$. Thus, $\widetilde\psi$ does not change sign in $\Omega$. 
	
	Let us show that this contradicts the strong maximum principle (or Harnack's inequality) for $L$. We consider the function $u$ on $\Omega\times\R$ defined by $u(x,t):=
	\widetilde\psi (x) e^{\sqrt{\lambda} t} $, which does not change sign and satisfies
	$$
	(L - \partial_t^2) u = 0
	\qquad\text{in}\ \Omega\times\R \,.
	$$
	Since the operator $L-\partial_t^2$ again satisfies the assumptions of the proposition, we can apply the strong maximum principle \cite[Corollaire 3.1]{Bo} (in the first case) or Harnack's inequality \cite[Theorem 5.3]{FrLa2} to deduce that $u\equiv 0$ in $\Omega\times\R$. Thus, $\widetilde\psi\equiv 0$, a contradiction.
\end{proof}

 
 \bibliographystyle{amsalpha}

\begin{thebibliography}{99}

\bibitem{A} L. Abatangelo, A. Ferrero and P. Luzzini.
\newblock On solutions to a class of degenerate equations with the Grushin operator.
\newblock J. Diff. Eq. 445 (2025), 113666.

\bibitem{Al} G. Alessandrini.
\newblock On Courant's nodal domain theorem.
\newblock Forum Math. 10 (1998), 521--532.

\bibitem{AKS}
 N. Aronszajn, A. Krzywicki and J. Szarski,
 \newblock  A unique continuation theorem for exterior differential forms on
Riemannian manifolds.
 \newblock Ark. Mat. 4 (1962), 417--453.

\bibitem{Ba} H. Bahouri.
\newblock Non prolongement unique des solutions d'op\'erateurs "somme de carr\'es".
\newblock Ann. Inst. Fourier 36 (1986), no. 4, 137--155.


\bibitem{BGM} A. Banerjee, N. Garofalo and R. Manna.
\newblock Carleman estimates for Baouendi-Grushin operators with applications to quantitative uniqueness and strong unique continuation.
\newblock Appl. Anal. 101 (2022), no. 10, 3667--3688.

\bibitem{Bath} M.S. Baouendi.
\newblock Sur une classe d'op\'erateurs elliptiques d\'eg\'en\'er\'es.
\newblock  Th\`ese d'\'etat. Universit\'e de Paris (Facult\'e des sciencs d'Orsay)  (1967).


\bibitem{Bao} M.S. Baouendi.
\newblock Sur une classe d'op\'erateurs elliptiques d\'eg\'en\'er\'es.
\newblock Bull. Soc. Math. France 95 (1967), 45--87.

\bibitem{BaGo} M.S. Baouendi and C. Goulaouic.
\newblock  R\'egularit\'e et th\'eorie spectrale pour une
classe d'op\'erateurs elliptiques d\'eg\'en\'er\'es.
\newblock Arch. Rat. Mec. Anal. 34 (1969), 361--379.

\bibitem{BerMe} P. B\'erard and D. Meyer.
\newblock In\'egalit\'es isop\'erim\'etriques et applications. 
\newblock Annales scientifiques de l'\'Ecole Normale Sup\'erieure, S\'erie 4,  Tome 15 (1982), no. 3,  513--541.

\bibitem{BiSo0} M. Sh. Birman and M. Z. Solomyak.
\newblock Quantitative Analysis in Sobolev Imbedding Theorems and Applications to Spectral Theory.
\newblock Amer. Math. Soc. Transl. Ser. 2, 114, American Mathematical Society, Providence, RI, 1980.

\bibitem{BiSo} M. Sh. Birman and M. Z. Solomyak.
\newblock Spectral asymptotics of nonsmooth elliptic operators. I.
\newblock Trans. Moscow Math. Soc. 27 (1972), 1--52.

\bibitem{BCP} P. Bolley, J. Camus, and Pham The Lai.
\newblock Noyau, r\'esolvante, et valeurs propres d'une classe d'op\'erateurs elliptiques et d\'eg\'en\'er\'es.
\newblock Lecture Notes in Mathematics  660 (1977), 33--46.

\bibitem{Bo} J.M. Bony.
\newblock Principe du maximum, in\'egalit\'e de Harnack et unicit\'e du probl\`eme de Cauchy pour des op\'erateurs elliptiques d\'eg\'en\'er\'es.
\newblock Annales de l'institut Fourier 19 (1969), no. 1, 277--304.

\bibitem{Bou} J. Bourgain.
\newblock On Pleijel's nodal domain theorem.
\newblock Int. Math. Res. Not. 2015, no. 6, 1601--1612.

\bibitem{CaFrLa} E. A. Carlen, R. L. Frank and S. Larson.
\newblock A Jensen inequality for partial traces and applications to partially semiclassical limits.
\newblock Lett. Math. Phys. 115 (2025), no. 3, Paper No. 52, 15 pp.

\bibitem{CX} J-Y. Chemin and Chao-Jian Xu.
\newblock Inclusions de Sobolev en calcul de Weyl-H\"ormander et
champs de vecteurs sous-elliptiques
\newblock Annales scientifiques de l'E.N.S. 4e s\'erie, 30 (1997), no. 6, 719--751.


\bibitem{CHT0} Y. Colin de Verdi\`ere, L. Hillairet and E. Tr\'elat.
\newblock Spectral asymptotics for sub-Riemannian  Laplacians, I: Quantum ergodicity and quantum limits in the 3-dimensional contact case.
\newblock Duke Math. Journal 167 (2018), no. 1, 109--174.

\bibitem{CHT} Y. Colin de Verdi\`ere, L.  Hillairet and E. Tr\'elat.
 \newblock Small-time asymptotics of hypoelliptic heat kernels near the diagonal, nilpotentization and related results.
 \newblock  Annales Henri Lebesgue 4 (2021), 897--971.

\bibitem{CHT1} Y.  Colin de Verdi\`ere, L. Hillairet  and E.  Tr\'elat.
 \newblock Spectral asymptotics for sub-Riemannian  Laplacians.
 \newblock Preprint (2022), arXiv:2212.02920.
 
 \bibitem{CDHT}  Y.  Colin de Verdi\`ere,  C. Dietze, M. V. de Hoop and 
E. Tr\'elat.
\newblock  Weyl formulae for some singular metrics with application
to acoustic modes in gas giants.
\newblock Preprint (2024), arXiv:2406.19734.

\bibitem{C} R. Courant.
\newblock Ein allgemeiner Satz zur Theorie der Eigenfunktionen selbstadjungierter Differentialausdr\"ucke.
\newblock G\"ott. Nachr. (1923), 81--84.

\bibitem{DPFaVi} N. De Ponti,  S. Farinelli and I. Y. Violo.
\newblock Pleijel nodal domain theorem in non-smooth setting.
\newblock Trans. Amer. Math. Soc. Ser. B 11 (2024), 1138--1182.

\bibitem{EL} S. Eswarathasan and C. Letrouit.
\newblock Nodal sets of Eigenfunctions of sub-Laplacians.
\newblock Internat. Math. Res. Notices (2023), no. 23, 20670--20700.

\bibitem{FrLa} B.~Franchi and E.~Lanconelli.
\newblock H\"older regularity theorem for a class of linear nonuniformly elliptic operators with measurable coefficients. 
\newblock Ann.~Scuola Norm.~Sup.~Pisa Cl.~Sci.~(4) \textbf{10} (1983), 523--541.

\bibitem{FrLa2} B.~Franchi  and  E.~Lanconelli.
\newblock An embedding theorem for Sobolev spaces related to non smooth vector fields and Harnack inequality. 
\newblock Comm. Partial Differential Equations \textbf{9} (1984), no. 13, 1237--1264.

 \bibitem{FH} R. Frank and B. Helffer.
 \newblock On Courant and Pleijel theorems for sub-Riemannian Laplacians.
 \newblock J. \'Ecole Polytechnique 12 (2025), 1083--1160. 
 
 \bibitem{Ga} N. Garofalo.
 \newblock Unique continuation for a class of elliptic operators which degenerate
 on a manifold of arbitrary codimension.
 \newblock J. Differential Equations 104 (1993), no. 1, 117--146.
 
\bibitem{GiMa} M.~Giaquinta and L.~Martinazzi.
\newblock An introduction to the regularity theory for elliptic systems, harmonic maps and minimal graphs.
\newblock Appunti. Sc. Norm. Super. Pisa (N. S.), 11. Edizioni della Normale, Pisa, 2012.
 
 \bibitem{Gr} V.V. Gru$\check{s}$in.
\newblock A certain class of hypoelliptic operators that are degenerate on a submanifold.
\newblock Mat. Sbornik 84 (126) (1971) 163-195.
 
 \bibitem{HS} A. Hassannezhad and D.  Sher. 
\newblock On Pleijel's nodal domain theorem for the Robin problem.
\newblock Bull. London Math. Soc. 56 (2024), no. 4, 1449--1467
 
\bibitem{HHOT} B. Helffer, T. Hoffmann-Ostenhof, S. Terracini.
\newblock Nodal domains and spectral minimal partitions.
\newblock Ann. Inst. H. Poincar\'e C Anal. Non Lin\'eaire 26 (2009), no. 1, 101--138.
 
\bibitem{HPS} B. Helffer and M. Persson Sundqvist.
\newblock On nodal domains in Euclidean balls.
\newblock Proc. Amer. Math. Soc. 144 (2017), no. 11, 4777--4791. 
 
\bibitem{LaWe} A. Laptev, T. Weidl,
\newblock Sharp Lieb--Thirring inequalities in high dimensions.
\newblock Acta Math. 184 (2000), no. 1, 87--111.
 
 \bibitem{Le} C. L\'ena.
\newblock Pleijel's nodal domain theorem for Neumann and Robin eigenfunctions. 
\newblock Annales de l'Institut Fourier 69 (2019), no. 1, 283--301. 

\bibitem{Lev} S. Levendorskii.
\newblock Degenerate elliptic equations.
\newblock Math. Appl., 258. Kluwer Academic Publishers Group, Dordrecht, 1993.

\bibitem{Ma} V. Maz'ya.
\newblock Sobolev spaces with applications to elliptic partial differential equations.
Second, revised and augmented edition.
\newblock Grundlehren Math. Wiss., 342, Springer, Heidelberg, 2011.

 \bibitem{Me} G. M\'etivier.
\newblock Fonction spectrale et valeurs propres d'une classe
d'op\'erateurs non elliptiques.
\newblock Comm. in PDE 1 (1976), 467--519.

\bibitem{Met2} G. M\'etivier.
\newblock Comportement asymptotique des valeurs propres d'op\'erateurs elliptiques d\'eg\'en\'er\'es.
\newblock Ast\'erisque 14-15 (1976)  215-249.

\bibitem{No} C. Nordin. 
\newblock  The asymptotic distribution of eigenvalues of a degenerate elliptic operator. 
\newblock Arkiv f\"or Mat. 10 (1972) 3--21. 
 
\bibitem{OK} P. Opic and A. Kufner.
\newblock Hardy--Type Inequalities.
\newblock Pitman Research Notes in Mathematics Series, vol. 219, Longman Scientific \& Technical, Harlow, 1990. 
 
 \bibitem{Pl} A.~Pleijel. 
\newblock Remarks on Courant's nodal theorem.
\newblock  Comm. Pure. Appl. Math. 9 (1956), 543--550. 

\bibitem{Po} I. Polterovich.
\newblock Pleijel's nodal domain theorem for free membranes.
\newblock Proc. Amer. Math. Soc. 137 (2009), no. 3, 1021--1024.

\bibitem{Qiu} Y. W. Qiu.
\newblock A note on the Pleijel theorem for H-type groups.
\newblock Preprint (2025), 	arXiv:2510.19381.

\bibitem{RoSi} D.~Robinson and A.~Sikora.
\newblock Analysis of degenerate elliptic operators of Gru\v{s}in type.
\newblock Math.~Z.~\textbf{260} (2008), no.~3, 475--508.

\bibitem{Ru} M. Rumin.
\newblock  Spectral density and Sobolev inequalities for pure and mixed states.
\newblock Geom. Funct. Anal. 20 (2010), no. 3, 817--844.

\bibitem{SR} I. Silvestre Rosello.
\newblock Weak Courant theorems for hypoelliptic operators.
\newblock In preparation.

\bibitem{Si} B. Simon.
\newblock Nonclassical eigenvalue asymptotics.
\newblock J. Funct. Anal. 53 (1983), no. 1, 84--98.

\bibitem{SiFI} B. Simon.
\newblock Functional Integration and Quantum Physics, 2nd edn. 
\newblock AMS Chelsea Publishing, Providence (2005)

\bibitem{Solo} I. A. Solomeshch.
\newblock Eigenvalues of some degenerate elliptic equations.
\newblock Mat. Sb. 54 (1961), no. 3, 295--310.

\bibitem{Stei} S. Steinerberger.
\newblock A geometric uncertainty principle with an application to Pleijel's estimate.
\newblock Ann. Henri Poincar\'e 15 (2014), no. 12, 2299--2319.

\bibitem{Ta} G. M. Tashchiyan.
\newblock Distribution of the eigenvalues of an elliptic Dirichlet problem. (Russian).
\newblock Proc. Sixth Drogobychsk. School-Sympos. Math. Programming and Related Questions, Sect. Functional Anal., 1974, 273--289.

\bibitem{Vi} M. I. Vishik.
\newblock Boundary-value problems for elliptic equations degenerating on the boundary of a region.
\newblock Mat. Sb. 35 (77) (1954), 513--568; English transl., Amer. Math. Soc. Transl. (2) 35 (1964), 15--78.

\bibitem{SV0}
  I. L. Vulis and M. Z. Solomjak.
  \newblock Spectral asymptotics of degenerate elliptic operators. 
  \newblock Soviet. Math. Dokl, vol 13, (1972), 1484-1488. 

\bibitem{SV} I. L. Vulis and M. Z. Solomyak.
\newblock Spectral asymptotics of second-order degenerate elliptic operators.
\newblock Mathematics of the USSR-Izvestiya 8 (1974), no. 6, 1343--1371.

\bibitem{Wa} K.  Watanabe.
\newblock  Sur l'unicit\'e du prolongement des solutions des \'equations elliptiques d\'eg\'en\'er\'ees.
\newblock Tohoku Math. Journ. 34(1982), 239--249.

\end{thebibliography}

\end{document}